\documentclass[12pt]{amsart}
\usepackage{amssymb}
\usepackage{hyperref}
\usepackage[margin=1.25in]{geometry}
\usepackage[utf8]{inputenc}
\usepackage{chngcntr}
\usepackage{apptools}
\usepackage{thm-restate}
\usepackage{url}
\AtAppendix{\counterwithin{theorem}{section}}
\usepackage{filecontents}

\newtheorem{theorem}{Theorem}[section]

\newtheorem{proposition}[theorem]{Proposition}
\newtheorem{lemma}[theorem]{Lemma}
\newtheorem{corollary}[theorem]{Corollary}
\newtheorem{claim}[theorem]{Claim}
\theoremstyle{remark}
\newtheorem{remark}[theorem]{Remark}
\theoremstyle{definition}
\newtheorem{definition}[theorem]{Definition}

\newcommand{\origin}{\mathfrak{o}}

\providecommand{\keyywords}[1]
{
  \small	
  \textbf{\textit{Keywords---}} #1
}

\newcommand{\R}{\mathbb{R}}  
\newcommand{\G}{\mathcal{G}_{n+1}}

\begin{document}


\title{Rotationally symmetric Ricci flow on $\mathbb{R}^{n+1}$}


\author{Francesco Di Giovanni}
\address{Department of Mathematics, University College London, Gower Street, London, WC1E 6BT, United Kingdom} 
\email{francesco.giovanni.17@ucl.ac.uk}





\begin{abstract}
We study the Ricci flow on $\mathbb{R}^{n+1}$, with $n\geq 2$, starting at some complete bounded curvature rotationally symmetric metric $g_{0}$. We first focus on the case where $(\mathbb{R}^{n+1},g_{0})$ does not contain minimal hyperspheres; we prove that if $g_{0}$ is asymptotic to a cylinder, then the solution develops a Type-II singularity and converges to the Bryant soliton after scaling, while if the curvature of $g_{0}$ decays at infinity, then the solution is immortal. As a corollary, we prove a conjecture by Chow and Tian about Perelman's standard solutions. We then consider a class of asymptotically flat initial data $(\mathbb{R}^{n+1},g_{0})$ containing a neck and we prove that if the neck is sufficiently pinched, in a precise way, the Ricci flow encounters a Type-I singularity. 
\end{abstract}
\maketitle
\keyywords{Ricci flow; Bryant Soliton; Standard solutions; Neckpinch.}

\section{Introduction}
Let $M$ be a smooth manifold. Given an initial metric $g_{0}$, Hamilton's Ricci flow is the evolution equation \cite{threemanifolds}
\[
\frac{\partial g}{\partial t} = - 2\,\text{Ric}_{g(t)}, \,\,\,\,\,\,\,\,\,\, g(0) = g_{0}.
\]
\noindent A Ricci flow solution $(M,g(t))_{0\leq t < T}$ whose maximal time of existence $T$ is finite can be classified as follows \cite{formationsingularities}:
\begin{align*}
\emph{Type-I}&:\,\,\,\,\,\sup_{M\times [0,T)}\lvert \text{Rm}_{g(t)}\rvert_{g(t)}(T-t) < \infty, \\
\emph{Type-II}&:\,\,\,\,\,\sup_{M\times [0,T)}\lvert \text{Rm}_{g(t)}\rvert_{g(t)}(T-t) = \infty.
\end{align*}
Since the Ricci flow is invariant under the action of the diffeomorphism group, it is natural to evolve metrics with symmetries; in this regard several properties of singular rotationally invariant Ricci flows have been analysed. 

In \cite{exampleneck},\cite{precyse} Angenent and Knopf constructed the first examples of finite time neckpinches by evolving a family of rotationally symmetric metrics on $S^{n+1}$ containing a stable minimal $n$-sphere. The first examples of Type-II singularities in dimension three or higher were produced by Gu and Zhu in \cite{type2}, where they studied a family of rotationally symmetric metrics on $S^{n+1}$. Later Angenent, Isenberg and Knopf proved that on $S^{n+1}$ there exist Ricci flows which behave like degenerate neckpinches, with the singularity modelled on the Bryant soliton \cite{degenerateneck}. Similarly, Wu found rotationally symmetric solutions on $\R^{n+1}$ which encounter a Type-II singularity and converge to the Bryant soliton \cite{wu}; however, Wu's argument only applies to initial data whose profile function converges uniformly to that of the Bryant soliton near the origin. 


In our first result we show that a large class of rotationally symmetric Ricci flows on $\R^{n+1}$ develop a finite time Type-II singularity modelled on the Bryant soliton \cite{bryant}.

\begin{theorem}\label{maintheoremnospheres}
Let $(\R^{n+1},g(t))_{0\leq t < T}$, with $n\geq 2$, be the Ricci flow solution evolving from a complete bounded curvature rotationally symmetric metric $g_{0}$. If $(\R^{n+1},g_{0})$ does not contain minimal hyperspheres and $g_{0}$ is asymptotic in $C^{0}$ to a round cylinder at infinity, then the solution develops a Type-II singularity at $T < \infty$ and converges to the Bryant soliton in the Cheeger-Gromov sense once suitably rescaled.
\end{theorem}

According to Theorem \ref{maintheoremnospheres} in the rotationally symmetric case the cylindrical behaviour at infinity determines the dynamics of the flow as long as there are no minimal hyperspheres. In light of this, one might expect that if instead the curvature of $g_{0}$ decays at infinity then the solution evolving from $g_{0}$ should be immortal. In this direction, Oliynyk and Woolgar proved that on $\R^{n+1}$ if $g_{0}$ is rotationally symmetric, \emph{asymptotically flat} and with no minimal embedded hyperspheres, then the Ricci flow solution starting at $g_{0}$ is immortal \cite{woolgar}. Our second result extends the long-time existence property in \cite{woolgar} to initial data that need not be close to the Euclidean metric outside a compact region.

\begin{theorem}\label{maintheoremnospheresimmortal}
Let $(\R^{n+1},g(t))_{0\leq t < T}$, with $n\geq 2$, be the Ricci flow solution evolving from a complete rotationally symmetric metric $g_{0}$ with curvature decaying at infinity. If $(\R^{n+1},g_{0})$ does not contain minimal hyperspheres, then the solution is immortal.
\end{theorem}
A simple application of Theorem \ref{maintheoremnospheres} and Theorem \ref{maintheoremnospheresimmortal} consists in classifying rotationally invariant Ricci flows with nonnegative bounded curvature.
\begin{corollary}\label{classificationpositive}
Let $(\R^{n+1},g(t))_{0\leq t < T}$, with $n\geq 2$, be the Ricci flow solution evolving from a complete rotationally symmetric metric $g_{0}$ with bounded nonnegative curvature. Then $T < \infty$ if and only if $g_{0}$ is asymptotic in $C^{0}$ to the round cylinder $\R \times S^{n}(r_{0})$ for some $r_{0} > 0$; in this case $T$ only depends on $n$ and $r_{0}$ and the solution develops a global Type-II singularity modelled on the Bryant soliton once suitably dilated.
\end{corollary}
Chen and Zhu already proved that any Ricci flow as in Corollary \ref{classificationpositive} develops a \emph{global} singularity at some $T$ only depending on the radius of the cylinder asymptotic to the initial metric \cite[Theorem A.1]{chenzhu} (see also Proposition \ref{propchenzhu}).

The classification of nonnegatively curved rotationally symmetric Ricci flows also leads to a better understanding of standard solutions. In the fundamental paper \cite{perelmansurgery} Perelman introduced the notion of \emph{standard} solutions on $\R^{3}$: they are obtained by evolving metrics constructed by gluing a hemispherical cap region to a round cylinder of scalar curvature one. Standard solutions were used to describe the behaviour of the flow after performing surgery.
In \cite{peng} Lu and Tian generalized Perelman's standard solutions; they considered Ricci flows on $\R^{n+1}$, with $n\geq 2$, starting at some rotationally symmetric metric with nonnegative curvature, sufficiently bounded geometry and asymptotic to a round cylinder (see Definition \ref{standardsolutions}). Chow and Tian have conjectured that any standard solution in the sense of \cite{peng} develops a Type-II singularity modelled on the Bryant soliton once suitably dilated \cite[Conjecture 1.2]{wu}. Wu gave evidence in favour of this conjecture by showing that there exist \emph{some} standard solutions converging to Bryant solitons \cite{wu}. 

Corollary \ref{classificationpositive} provides an affirmative answer to the Chow-Tian conjecture.
\begin{corollary}[Chow-Tian Conjecture]\label{conjecture} Sequences of appropriately scaled standard solutions $\emph{(}$as defined in \cite{peng}$\emph{)}$ with marked origins converge to Bryant solitons in a suitable sense.
\end{corollary}
A three dimensional version of this result was proved by Ding in \cite{ding}. 

We point out that the notion of standard solutions discussed in \cite{peng} does not require the curvature to attain a maximum. If instead we also assume a pinching condition where the radial sectional curvature $K_{g_{0}}$ is bounded from above by the spherical sectional curvature $L_{g_{0}}$, then Corollary \ref{classificationpositive} and the recent classification of (rotationally symmetric) $\kappa$-solutions obtained by Brendle in \cite{brendle} and Li and Zhang in \cite{brendle2} guarantee that by blowing up the standard solution at the origin $\origin\in\R^{n+1}$ along \emph{any} time sequence approaching the maximal time one obtains the Bryant soliton in the limit.
\\

\begin{corollary}\label{corollarybrendlino}
Let $(\R^{n+1},g(t))_{0\leq t < T}$ be a standard solution to the Ricci flow $($as defined in \cite{peng}$)$ starting at $g_{0}$. If $K_{g_{0}}\leq L_{g_{0}}$, then for any $t_{j}\nearrow T$ the rescaled standard solutions $(\mathbb{R}^{n+1},g_{j}(t),\origin)$ defined on $[-R_{g(t_{j})}(\origin)t_{j},0]$ by $g_{j}(t)\doteq R_{g(t_{j})}(\origin)g(t_{j}+t/R_{g(t_{j})}(\origin))$ converge to the Bryant Soliton (up to scaling).  
\end{corollary}

According to \cite{exampleneck} the conclusion of Theorem \ref{maintheoremnospheresimmortal} should generally fail if $(\R^{n+1},g_{0})$ contains minimal hyperspheres. In fact, in \cite{woolgar} it was expected that if $g_{0}$ is asymptotically flat and contains minimal embedded hyperspheres forming a neck region which is sufficiently pinched, then the Ricci flow starting at $g_{0}$ develops a Type-I singularity caused by the radius of the neck going to zero in finite-time. Conversely, if the pinching is mild, then the neck should disappear in finite time and the flow should hence be immortal. By extending the analysis in \cite{exampleneck} to $\R^{n+1}$ we are able to confirm such expectation. We consider the Ricci flow evolving from an asymptotically flat metric $g_{0}$ of the form 
\[
g_{0} = \xi_{0}^{2}(x)dx\otimes dx + \phi_{0}^{2}(x)\hat{g},
\]
\noindent where $\hat{g}$ is the constant curvature one metric on $S^{n}$; we say that $g_{0}$ has a \emph{neck region} if $\phi_{0}$ has a local maximum at some radial coordinate $x_{\ast}$ and a local minimum at some radial coordinate $y_{\ast}> x_{\ast}$. Following \cite{precyse} the pinching of the neck is then given by the ratio (difference) between the radii $\phi_{0}(x_{\ast})$ and $\phi_{0}(y_{\ast})$. 

The statement below is a weaker version of our result and we refer to Theorem \ref{asymptoticsscrittobene} for a complete statement containing the cylindrical asymptotics.
\begin{theorem}\label{mainresultneckpinches}
Let $(\R^{n+1},g(t))_{0\leq t < T}$, with $n\geq 2$, be the solution to the Ricci flow evolving from an asymptotically flat rotationally symmetric metric $g_{0}$ containing a neck region $(x_{\ast},y_{\ast})\times S^{n}$. Assume that $\emph{Ric}_{g_{0}} > 0$ on the closed Euclidean ball $B(\origin,x_{\ast})$ and that $R_{g_{0}}\geq 0$ on $\R^{n+1}$. Let $\beta$ be defined as
\[
\beta \doteq \inf_{\R^{n+1}}\phi_{0}^{2}(L_{g_{0}}-K_{g_{0}}),
\]
\noindent and let $r > 0$ satisfy
\[ 
r^{2} > \frac{n+1-2\beta}{n-1} + 1.
\]
\noindent If $\phi_{0}(x_{\ast})\geq r \phi_{0}(y_{\ast})$, then the solution develops a Type-I singularity which is modelled on a family of shrinking cylinders.
\end{theorem}
Finally, we show that there exist examples of necks that disappear in finite time along the Ricci flow; the next result follows by combining the adaptation of \cite{exampleneck} to $\R^{n+1}$ and the stability result for the Euclidean metric proved by Schn\"urer, Schulze and Simon in \cite{schulze}.
\begin{proposition}\label{propositionlercia}
There exists $\varepsilon_{0} = \varepsilon_{0}(n) > 0$ such that if $g_{0}$ is an asymptotically flat rotationally symmetric metric which has a neck and is $\varepsilon_{0}$-close to the Euclidean metric on $\R^{n+1}$, then the maximal Ricci flow solution $g(t)$ evolving from $g_{0}$ is immortal and the neck disappears in finite time.
\end{proposition}
\subsection*{Outline.} In Section 2 we discuss some preliminaries and we prove a few basic estimates. In Section 3 we analyse rotationally invariant Ricci flows on $\R^{n+1}$ with no minimal embedded hyperspheres. We show that by \cite{nodal} no minimal hyperspheres appear along the flow and that the curvature is controlled via lower bounds for the radius $\phi$ as in \cite{exampleneck}. In Section 4 we prove that under the assumptions of Theorem \ref{maintheoremnospheres} the flow develops Type-II singularities; the main ingredients are the characterization of Type I flows in \cite{type1} and the classification of conformally flat shrinkers in \cite{zhang}. The appearance of the Bryant soliton follows from \cite{eternal} and \cite{cao}, or alternatively from the recent classification of (rotationally symmetric) $\kappa$-solutions in \cite{brendle} and \cite{brendle2}. We then show that in the setting of Theorem \ref{maintheoremnospheresimmortal} any singularity model has positive asymptotic volume ratio; according to \cite{pseudolocality}, the latter property implies that any Ricci flow as in Theorem \ref{maintheoremnospheresimmortal} is immortal. Section 5 is devoted to classifying nonnegatively curved rotationally symmetric Ricci flows on $\R^{n+1}$, with focus on studying the singularities of standard solutions. In Section 6 we extend the analysis in \cite{exampleneck} to $\R^{n+1}$ to prove Theorem \ref{mainresultneckpinches} (restated in Theorem \ref{asymptoticsscrittobene}); we also outline how the examples of initial data constructed in \cite{exampleneck} may be modified to provide analogous initial data for which Theorem \ref{mainresultneckpinches} applies. We derive cylindrical asymptotics for the neckpinch following \cite{IKS1}. Finally, using \cite{schulze} we provide examples of initial data with necks that evolve to metrics with no minimal embedded hyperspheres in finite time.

\subsection*{Acknowledgements.} The author would like to thank his advisor Jason Lotay for his mentorship and constant support and for many helpful conversations.
\section{Preliminaries}
Let $n \geq 2$ be an integer. Away from the origin, any rotationally symmetric metric on $\R^{n+1}$ is of the form 
\begin{equation}\label{initialmetricwithxi}
g = \xi^{2}(x)\,dx\otimes dx + \phi^{2}(x)\,\hat{g}
\end{equation}
\noindent where $\hat{g}$ is the standard metric of constant curvature one on $S^{n}$ and $\xi,\phi$ are smooth functions on $(0,+\infty)$. Once we introduce the geometric coordinate $s$ representing the $g$-distance from the origin, it is a general fact that $g$ extends smoothly to the origin if and only if 
\begin{equation}\label{boundaryconditions}
\lim_{s\rightarrow 0} \frac{d^{2k}\phi}{ds^{2k}}(s) = 0, \,\,\,\,\,\,\,\,\,\, \lim_{s\rightarrow 0}\frac{d\phi}{ds}(s) = 1,
\end{equation}
\noindent for any integer $k\geq 0$. From now on we assume that \eqref{boundaryconditions} is satisfied; in particular we can write $g$ as
\begin{equation}\label{rotationallysymmetrictime0}
g = ds\otimes ds + \phi^{2}(s)\,\hat{g}.
\end{equation} 
\noindent 
\noindent In the following we always regard $\phi = \phi(s) = \phi(s(x))$ as a function of the variable $x$ (and of time for solutions to the Ricci flow); the spatial derivative with respect to $s$ is therefore intended to be the vector field 
\begin{equation}\label{changevariable}
\partial_{s} = \frac{1}{\xi(x)}\partial_{x}.
\end{equation}
\noindent We adopt the same notations as in \cite{exampleneck}; for any metric $g$ of the form \eqref{rotationallysymmetrictime0} we denote the sectional curvatures of the 2-planes perpendicular to the fibers $\{x\}\times S^{n}$ and of the 2-planes tangential to these fibers by $K$ and $L$ respectively. From the rotational symmetry it follows that the curvature of $g$ is entirely described by $K$ and $L$ which are given by 
\begin{equation}\label{formulasforKandL}
K = -\frac{\phi_{ss}}{\phi}, \,\,\,\,\,\,\,\,\,\, L = \frac{1 - \phi_{s}^{2}}{\phi^{2}}.
\end{equation} 
\noindent By tracing we get the formulas for the Ricci tensor and the scalar curvature:
\begin{align}\label{Riccitensor}
\text{Ric}_{g} &= -n\frac{\phi_{ss}}{\phi}(ds)^{2} + \left(-\phi\phi_{ss} + (n-1)(1 - \phi_{s}^{2}) \right)\hat{g}, \\  R_{g} &= n\left(-2 \frac{\phi_{ss}}{\phi} + (n-1)\frac{1-\phi_{s}^{2}}{\phi^{2}} \right). \label{scalarcurvature}
\end{align}
\subsection{Derived equations.}
Let $g_{0}$ be a complete rotationally symmetric metric on $\R^{n+1}$ of the form \eqref{rotationallysymmetrictime0}. If $g_{0}$ has bounded curvature then there exists a solution $g(t)$ to the Ricci flow starting at $g_{0}$ \cite{shi}; moreover, this solution is unique in the class of complete solutions with bounded curvature on compact subintervals \cite{uniqueness}. By the Ricci flow diffeomorphism invariance and the uniqueness result in \cite{uniqueness} such solution preserves the rotational symmetry; therefore, we may write $g(t)$ as
\begin{equation}\label{ricciflowsolution}
g(t) = \xi^{2}(x,t)\,dx\otimes dx + \phi^{2}(x,t)\,\hat{g} = ds\otimes ds + \phi^{2}(s,t)\,\hat{g},
\end{equation}
\noindent where $s = s(x,t)$ is the time-dependent $g(t)$-distance from the origin. 
\noindent From \eqref{Riccitensor} we derive the evolution equations for $\xi$ 
\begin{equation}\label{evolutionxi}
\xi_{t} = n\frac{\phi_{ss}}{\phi}\xi
\end{equation}
\noindent and for the radius $\phi$
\begin{equation}\label{equationphi} 
\phi_{t} = \phi_{ss} - (n-1)\frac{1-\phi_{s}^{2}}{\phi}.
\end{equation}
\noindent Since the geometric variable $s$ depends on time, we have a nonvanishing commutator between $\partial_{s}$ and $\partial_{t}$; by \eqref{evolutionxi} we get
\begin{equation}\label{commutatorformula}
\left[\partial_{t},\partial_{s}\right] = \left[\partial_{t},\frac{\partial_{x}}{\xi(x,t)}\right] = -(\text{log}\,\xi)_{t}\partial_{s} = - n\frac{\phi_{ss}}{\phi}\partial_{s}.
\end{equation}
\noindent Using the commutator formula and \eqref{equationphi} we compute the equations for the first derivative of $\phi$
\begin{equation}\label{equationfirstderivative}
(\phi_{s})_{t} = (\phi_{s})_{ss} + \frac{n-2}{\phi}\phi_{s}(\phi_{s})_{s} + (n-1)\frac{1-\phi_{s}^{2}}{\phi^{2}}\phi_{s}
\end{equation}
\noindent and for its second derivative
\begin{equation}\label{equationsecondderivative}
(\phi_{ss})_{t} = (\phi_{ss})_{ss} + (n-2)\frac{\phi_{s}}{\phi}(\phi_{ss})_{s} - 2\frac{\phi_{ss}^{2}}{\phi} -(4n -5)\frac{\phi_{s}^{2}}{\phi^{2}}\phi_{ss} + \frac{n-1}{\phi^{2}}\phi_{ss} -2(n-1)\frac{\phi_{s}^{2}(1-\phi_{s}^{2})}{\phi^{3}}. 
\end{equation}
\noindent Similarly to \cite{exampleneck} we introduce the quantity 
\begin{equation}\label{definitionA}
A = \phi^{2}(L - K) = \phi\phi_{ss} + 1 -\phi_{s}^{2},
\end{equation}
\noindent which is a scale-invariant measure of the difference between the spherical sectional curvature $L$ and the radial sectional curvature $K$. 
From \cite[Lemma 3.1]{exampleneck} it follows that the quantity $A$ evolves by 
\begin{equation}\label{evolutionA}
A_{t} = A_{ss} + (n-4)\frac{\phi_{s}}{\phi}A_{s} -4(n-1)\frac{\phi_{s}^{2}}{\phi^{2}}A.
\end{equation}
\noindent We also write the expression for the Laplacian along the flow: for any smooth radial function $f$ the Laplacian associated with the solution to the Ricci flow at time $t$ is given by 
\begin{equation}\label{formulalaplacian}
\Delta f = f_{ss} + n\frac{\phi_{s}}{\phi}f_{s}.
\end{equation}
\subsection{Basic estimates.} We dedicate the end of this section to proving general bounds for rotationally invariant Ricci flows on $\R^{n+1}$. We first show that we can control the curvature of the Ricci flow solution via lower bounds for the radius $\phi$; the following property is analogous to \cite[Lemma 7.1]{exampleneck}. 
\begin{lemma}\label{crucialestimateinterior}
Let $(\R^{n+1},g(t))_{0\leq t < T}$ be the maximal Ricci flow solution evolving from a complete bounded curvature rotationally symmetric metric $g_{0}$. Let $U\subset \R^{n+1}$ and assume that $\phi^{2}\lvert \emph{Rm}_{g(t)}\rvert_{g(t)}\leq C$ along the parabolic boundary of $U\times [0,T)$ for some $C > 0$. Then $\phi^{2}\lvert \emph{Rm}_{g(t)}\rvert_{g(t)}\leq C^{\prime}$ in $U\times [0,T)$ for some $C^{\prime}\in [C,\infty)$.
\end{lemma}
\begin{proof}
It suffices to show that $\phi^{2}(\lvert L\rvert + \lvert K \rvert)\leq C^{\prime}$ in $U\times [0,T)$. Since $\phi^{2}L = 1-\phi_{s}^{2}$ is uniformly bounded along the parabolic boundary of $U\times [0,T)$, by the evolution equation \eqref{equationfirstderivative} we deduce that $\phi_{s}$ cannot diverge in $U\times [0,T)$ along a sequence of interior maxima (minima).
\\ We now consider the quantity $A$ defined in \eqref{definitionA}. By assumption $A$ is controlled along the parabolic boundary of $U\times [0,T)$; from \eqref{evolutionA} we get that $A$ is hence bounded in $U\times [0,T)$. We may thus conclude that $\phi^{2}K = \phi^{2}L - A$ is uniformly bounded in $U\times [0,T)$, which completes the proof.
\end{proof}
When the scale-invariant estimate in Lemma \ref{crucialestimateinterior} is satisfied on a given region as long as the solution exists then we can always define a limit (possibly degenerate) radius.
\begin{lemma}\label{existlimit}
Let $(\R^{n+1},g(t))_{0\leq t < T}$ be the maximal Ricci flow solution evolving from a complete bounded curvature rotationally symmetric metric $g_{0}$. Assume that $T < \infty$ and that $\phi^{2}\lvert \emph{Rm}_{g(t)}\rvert_{g(t)}\leq C$ on $U\times [0,T)$, where $U\subset \R^{n+1}$. Then for any $p\in U$ the limit $\lim_{t\nearrow T}\phi(p,t)$ exists finite. 
\end{lemma}
\begin{proof}
From \eqref{equationphi} we derive that for any $p\in U$ we have
\[
\lvert \partial_{t}(\phi^{2})\rvert (p) = \lvert 2\phi\phi_{ss} -2(n-1)(1-\phi_{s}^{2})\rvert(p) = \lvert-2\phi^{2}K -2(n-1)\phi^{2}L\rvert(p) \leq C.
\]
\noindent Therefore the function $\phi(p,\cdot)$ is Lipschitz in $[0,T)$ and the conclusion follows.
\end{proof}
We finally prove that lower bounds for the scale-invariant quantity $A$ defined in \eqref{definitionA} are preserved along the Ricci flow. In the following $\origin$ denotes the origin of $\R^{n+1}$.
\begin{lemma}\label{lowerboundApersists}
Let $(\R^{n+1},g(t))_{0\leq t < T}$ be the maximal Ricci flow solution evolving from a complete bounded curvature rotationally symmetric metric $g_{0}$. If $A(\cdot,0)\geq -\beta$, for some $\beta \geq 0$, then $A(\cdot,t) \geq -\beta$ for any $t\in [0,T)$.
\end{lemma}
\begin{proof}
We first check that the radius $\phi$ has a positive lower bound away from the origin.
\begin{claim}\label{phiboundedawayfromzero}
For any $x_{0} > 0$ and $t_{0} < T$ there exists $\delta(x_{0},t_{0}) > 0$ such that $\phi \geq \delta$ in $(\R^{n+1}\setminus B(\origin,x_{0}))\times [0,t_{0}]$. 
\end{claim}
\begin{proof}[Proof of Claim \ref{phiboundedawayfromzero}]
Let $\alpha_{0} \doteq \sup\lvert \text{Rm}_{g_{0}}\rvert_{g_{0}}$. Given $x_{0} > 0$, if $\phi^{2}(x,0) \leq (2\alpha_{0})^{-1}$ for all $x\geq x_{0}$ then we have
\[
\lvert 1 - \phi_{s}^{2}\rvert(x,0) \leq \alpha_{0}\phi^{2}(x,0)\leq \frac{1}{2},
\]
\noindent which then implies that $\phi(x,0)\rightarrow \infty$, but that is not possible. Therefore, we deduce that there exists a sequence $p_{j}\rightarrow \infty$ such that $\phi(p_{j},0) > (2\alpha_{0})^{-1}$. Assume for a contradiction that there exists a sequence $q_{j}\rightarrow \infty$ such that $\phi(q_{j},0) \leq (2\alpha_{0})^{-1}$. After reordering the sequences, we derive that there exists a sequence of local minima $\tilde{q_{j}}\rightarrow \infty$ such that $\phi(\tilde{q_{j}},0) \leq (2\alpha_{0})^{-1}$; it follows that
\[
\lvert 1 - \phi_{s}^{2} \rvert (\tilde{q_{j}},0) \equiv 1 \leq \alpha_{0}\phi^{2}(\tilde{q_{j}},0) \leq \frac{1}{2}.
\]  
\noindent We conclude that $\phi(x,0) \geq \delta > 0$ for some $\delta > 0$. Finally, we note that given $t_{0} < T$ standard distortion estimates of the curvature imply that $\phi$ has a time-dependent positive lower bound as claimed.
\end{proof}
For any $t < T$ there exists $\alpha(t) > 0$ such that $\lvert \phi_{ss} \rvert \leq \alpha \phi$; thus $\phi$ and $\phi_{ss}$ are exponentially bounded, which implies that $\phi_{s}$ is exponentially bounded. In particular, given $t_{0} < T$ there exist $M = M(t_{0})$ and $\alpha =\alpha(t_{0})$ such that $\lvert A(s,t)\rvert \leq M \exp(\alpha s)$ for any $t\in [0,t_{0}]$. Let $\varepsilon$, $\eta$ and $\gamma$ be positive constants to be chosen below; we define the lower barrier
\begin{equation}\label{definitionW}
(s,t)\mapsto W(s,t) \doteq \varepsilon \exp\left( \frac{s^{2}}{1- \eta t} + \gamma t \right).
\end{equation}
\noindent Using \eqref{evolutionA} we can write the evolution equation of $\hat{A}\doteq A + \beta + W$ for $0\leq t \leq 1/2\eta$ as
\begin{align*}
\hat{A}_{t} &\geq (\hat{A})_{ss} + (n-4)\frac{\phi_{s}}{\phi}(\hat{A})_{s} -4(n-1)\frac{\phi_{s}^{2}}{\phi^{2}}\hat{A} \\ &+ \frac{W}{(1-\eta t)^{2}}\left( s^{2}(\eta - 4) + \gamma(1-\eta t)^{2} + (1-\eta t)(2s\partial_{t}s -2 -(n-4)\frac{\phi_{s}}{\phi}2s) \right).
\end{align*}
\noindent By distortion estimates of the distance function there exists $C = C(t_{0})$ such that $\partial_{t} s \geq - Cs$ in $\R^{n+1}\times [0,\min\{t_{0},(2\eta)^{-1}\}]$. From the boundary conditions we derive that there exists a neighbourhood of the origin where $s\phi_{s}/\phi$ is uniformly bounded. By the boundedness of the curvature we get $\phi_{s}^{2} \leq 1 + C(t_{0}) \phi^{2}$ in $\R^{n+1}\times [0,\min\{t_{0},(2\eta)^{-1}\}]$; thus, by Claim \ref{phiboundedawayfromzero} we deduce that away from the origin the following estimate is satisfied 
\[
\left(\frac{\phi_{s}}{\phi}\right)^{2}\leq C(t_{0}) + \frac{1}{\phi^{2}}\leq C(t_{0}) + \frac{1}{\delta^{2}}.
\] Therefore $s\phi_{s}/\phi\leq C(t_{0})(1 + s)$ in $\R^{n+1}\times [0,\min\{t_{0},(2\eta)^{-1}\}]$. We conclude that we can always pick $\eta = \eta(t_{0})$ and $\gamma = \gamma(t_{0})$ such that in $\R^{n+1}\times [0,\min\{t_{0},(2\eta)^{-1}\}]$ we have
\[
\hat{A}_{t} > (\hat{A})_{ss} + (n-4)\frac{\phi_{s}}{\phi}(\hat{A})_{s} -4(n-1)\frac{\phi_{s}^{2}}{\phi^{2}}\hat{A}.
\]
\noindent Since $A$ is exponentially bounded and $A(\origin,t) = 0$ we see that if $\hat{A} < 0$ somewhere in $\R^{n+1}\times [0,\min\{t_{0},(2\eta)^{-1}\}]$, then there exist $z > 0$ sufficiently small, $\bar{x}$ and $\bar{t}$ such that $\hat{A}(\cdot,\bar{t})$ has a negative minimum at $\bar{x}$ where $\hat{A}(\bar{x},\bar{t}) = - z$ for the first time; however, we have shown that $\hat{A}_{t}(\bar{x},\bar{t}) > 0$, which gives a contradiction. We then obtain that $A(\cdot,t) \geq -\beta$ for any $t\in [0,\min\{t_{0},(2\eta)^{-1}\}]$ once we let $\varepsilon$ in \eqref{definitionW} go to zero. We may finally iterate the step and conclude that $A$ remains bounded from below by $-\beta$ in $\R^{n+1}\times [0,t_{0}]$; since $t_{0} < T$ was arbitrary, the proof is complete.
\end{proof}
\section{Analysis of Ricci flow with no minimal hyperspheres}
We consider the maximal Ricci flow solution $(\R^{n+1},g(t))_{0\leq t < T}$ evolving from a complete bounded curvature rotationally symmetric metric $g_{0}$ such that $(\R^{n+1},g_{0})$ does not contain minimal embedded hyperspheres. We first check that the last condition persists in time; namely, we show that minimal hyperspheres cannot appear along the Ricci flow solution if none existed at the initial time. Since the hypersphere of radius $x$ is minimal in $(\R^{n+1},g(t))$ when $\phi_{s}(x,t) = 0$, similarly to \cite{exampleneck} we derive the control on the formation of minimal hyperspheres from applying the Sturmian theorem to the evolution equation of $\phi_{s}$. In the following estimates $C$ always denotes a uniform constant that may change from line to line while $\origin$ denotes the origin of $\R^{n+1}$.
\begin{lemma}\label{nominimalspheres}
Let $(\R^{n+1},g(t))_{0\leq t < T}$ be the Ricci flow solution evolving from a complete bounded curvature rotationally symmetric metric $g_{0}$. If $\phi_{s}(\cdot,0) > 0$ then $\phi_{s}(\cdot,t) > 0$ for any $t\in[0,T)$.
\end{lemma}
\begin{proof}
Using \eqref{scalarcurvature} and \eqref{formulalaplacian} we can write the evolution equation of $\phi_{s}$ as 
\[
(\phi_{s})_{t} = \Delta \phi_{s} + \frac{R_{g(t)}}{n}\phi_{s}.
\]
\noindent Assume for a contradiction that $\phi_{s}(p_{0},t_{0}) < 0$ at some space-time point; by \cite{shi} given $T^{\prime}\in (t_{0},T)$ there exists $C = C(T^{\prime})$ such that $\lvert R_{g(t)}(\cdot)\rvert \leq n C$ in $\R^{n+1}\times [0,T^{\prime}]$. Therefore, at any space-time point in $\R^{n+1}\times [0,T^{\prime}]$ where $\phi_{s}$ is negative the time derivative $(\phi_{s})_{t}$ satisfies
\begin{equation}\label{maxprinciplephis}
(\phi_{s})_{t} \geq \Delta \phi_{s} + C\phi_{s}.
\end{equation}
\noindent We have already seen in the proof of Lemma \ref{lowerboundApersists} that $\lvert \phi_{s}(p,t)\rvert \leq \text{exp}(C(d_{g(t)}(\origin,p) + 1))$ for any $(p,t)\in\R^{n+1}\times [0,T^{\prime}]$.  
We can apply the maximum principle to \eqref{maxprinciplephis} and conclude that $\phi_{s}(\cdot,t)\geq 0$ in $\R^{n+1}\times [0,T)$ because $T^{\prime} < T$ was arbitrary.
\\ In fact, the inequality is strict for all positive times. Indeed, if there exist $p_{0}\in\R^{n+1}$ and $t_{0} > 0$ such that $\phi_{s}(p_{0},t_{0}) = 0$, then by the previous derivations we see that $p_{0}$ must be a minimum point for $\phi_{s}(\cdot,t_{0})$: the strong maximum principle implies that $\phi_{s}$ must vanish in the space-time region $\R^{n+1}\times [0,t_{0}]$, thus violating the boundary conditions \eqref{boundaryconditions}.
\end{proof}
Next we need to control the curvature of the Ricci flow solution only in terms of lower bounds for $\phi$. For if the latter condition holds, then by \eqref{boundaryconditions} and Lemma \ref{nominimalspheres} any solution with a nonempty singular set must in particular become singular around the origin; this geometric property guarantees that there always exist singularity models that are not shrinking cylinders. Thus in the following we identify under which assumptions on the behaviour of $g_{0}$ at spatial infinity the evolving solution $g(t)$ satisfies the estimate
\begin{equation}\label{crucialestimate}
\phi^{2}\lvert \text{Rm}_{g(t)}\rvert_{g(t)}\leq C,
\end{equation}
\noindent for some uniform constant $C = C(g_{0})$. The strategy consists in proving that such bound holds outside a sufficiently large ball and then using Lemma \ref{crucialestimateinterior} to deduce that the same control must extend to the ball. 
\\Let $g_{0}$ be a complete bounded curvature rotationally symmetric metric without minimal hyperspheres; the radius $\phi(\cdot,0)$ is hence increasing and admits a limit at infinity.
\subsection{Ricci flow with bounded radius.} When the limit $\lim_{x\rightarrow \infty}\phi(x,0)$ is finite, the initial metric is asymptotic in $C^{0}$ to a round cylinder; in this case the solution to the Ricci flow is controlled by a shrinking cylinder at infinity.
\\
\begin{lemma}\label{controlfinitelimit}
Let $(\R^{n+1},g(t))_{0\leq t < T}$ be the maximal Ricci flow solution evolving from a complete bounded curvature rotationally symmetric metric $g_{0}$ without minimal embedded hyperspheres. If $g_{0}$ is asymptotic in $C^{0}$ to the round cylinder of radius $r$ at infinity, then 
\begin{itemize}
\item[(i)] The solution becomes singular at a finite time satisfying $2T(n-1) \leq r^{2}$.
\item[(ii)] There exists $C > 0$ such that $\phi^{2}\lvert \emph{Rm}_{g(t)}\rvert_{g(t)}\leq C$ in $\R^{n+1}\times [\frac{T}{2},T)$.
\end{itemize} 
\end{lemma}
\begin{proof} The evolution equation of $\phi^{2}$ is given by
\[
\phi^{2}_{t} = \Delta\phi^{2} -4\phi_{s}^{2} -2(n-1).
\]
\noindent By a standard application of the maximum principle we find $\phi^{2}(x,t)\leq r^{2} - 2(n-1)t$ as long as the solution exists, which then implies (i).
\\In order to prove (ii) we consider $t\in [T/2,T)$. From the estimate above and Lemma \ref{nominimalspheres} we derive that $\phi(x,t)$ admits a positive finite limit as $x\rightarrow\infty$, which also implies that $\phi_{s}$ is integrable in $(0,\infty)$, once we regard $\phi_{s}$ as a function of $s$. Since the curvature is bounded at time $t$ by some constant $C$ we deduce that $\lvert \phi_{ss}\rvert = \lvert K \rvert \phi\leq C \phi \leq C$; thus we find $\phi_{s}(x,t)\rightarrow 0$ at infinity. Therefore $\phi_{s}$ is uniformly controlled at the origin and at spatial infinity in $[T/2,T)$; the same argument in Lemma \ref{crucialestimateinterior} shows that $\lvert\phi_{s}\rvert\leq C$ in $\R^{n+1}\times [T/2,T)$.
\\Similarly, by Shi's derivative estimates, the Koszul formula and the uniform bound on $\phi_{s}$ we obtain 
\[
 \lvert\phi_{sss}\rvert \leq \phi\lvert K_{s}\rvert + \lvert K\rvert \lvert\phi_{s}\rvert\leq C(\lvert\nabla \text{Rm}_{g(t)}\rvert + C \leq C,
 \]
\noindent which then implies that $\phi_{ss}(x,t)\rightarrow 0$ as $x\rightarrow \infty$, being the integral of $\phi_{ss}(\cdot,t)$ on $(0,\infty)$ convergent. We conclude that $\phi^{2}\lvert \text{Rm}_{g(t)}\rvert_{g(t)}$ is uniformly controlled at the origin and at spatial infinity for any $t\in [T/2,T)$; we may then apply Lemma \ref{crucialestimateinterior}.  
\end{proof} 
We finally consider a subclass of solutions with bounded radius defined by requiring $g_{0}$ to further satisfy the following scale-invariant pinching condition:
\begin{equation}\label{pinchingcondition}
A(\cdot,0) \equiv \phi^{2}(L-K)(\cdot,0) = (\phi_{ss}\phi + 1 - \phi_{s}^{2})(\cdot,0) \geq 0.
\end{equation}
\noindent The constraint \eqref{pinchingcondition} implies that $g_{0}$ has a tip located at the origin; moreover, this tip persists along the solution evolving from $g_{0}$.
\begin{lemma}\label{tipregion}
Let $(\R^{n+1},g(t))_{0\leq t < T}$ be the maximal Ricci flow solution evolving from a complete bounded curvature rotationally symmetric metric $g_{0}$ without minimal embedded hyperspheres. Assume that $A(\cdot,0)\geq 0$ and that $g_{0}$ is asymptotic in $C^{0}$ to a round cylinder at infinity. Then for any $p\in\R^{n+1}$ and $t\in [0,T)$ the following holds
\[
R_{g(t)}(\origin)\geq R_{g(t)}(p).
\]
\end{lemma}
\begin{proof} 
By Lemma \ref{lowerboundApersists} we know that $A(\cdot,t)\geq 0$ along the flow. For any $x > 0$ and for any $t\in [0,T)$ we have
\[
\partial_{s}\left(\frac{1-\phi_{s}^{2}}{\phi^{2}}\right)(x,t) = -2\frac{\phi_{s}}{\phi^{3}}A(x,t) \leq 0,
\]
\noindent where we have also used that $\phi_{s}(\cdot,t)\geq 0$. From \eqref{boundaryconditions} we derive that $\phi_{sss}(\origin,t)$ exists finite for any $t\in[0,T)$; we may thus apply l'H\^{o}pital's rule to find that
\begin{align*}
R_{g(t)}(\origin) &= n(n+1)(-\phi_{sss})(\origin,t) \\ &= n(n+1)\lim_{y\rightarrow 0}\left(\frac{1-\phi_{s}^{2}}{\phi^{2}}\right)(y,t) \geq n(n+1) \left(\frac{1-\phi_{s}^{2}}{\phi^{2}}\right)(x,t)
\end{align*}
\noindent for any $(x,t)\in (0,\infty)\times [0,T)$. Therefore, since the condition $A(\cdot,t) \geq 0$ also implies $-\phi_{ss}/\phi \leq (1-\phi_{s}^{2})/\phi^{2}$ in $\R^{n+1}\setminus\{\origin\}\times [0,T)$, we finally derive
\[
R_{g(t)}(p) = n\left(-2 \frac{\phi_{ss}}{\phi} + (n-1)\frac{1-\phi_{s}^{2}}{\phi^{2}} \right)(p,t) \leq n(n+1)\left(\frac{1-\phi_{s}^{2}}{\phi^{2}}\right)(p,t) \leq R_{g(t)}(\origin).
\]
\end{proof}
\subsection{Ricci flow with unbounded radius.} When the radius $\phi(x,0)$ diverges as $x\rightarrow \infty$ we generally have a weaker understanding of the geometry at infinity; for example, there exist initial data with exponential volume growth where the curvature stays away from zero at infinity as in the bounded radius case but the scale invariant quantity appearing in \eqref{crucialestimate} diverges at spatial infinity. In order to avoid such cases, we require the curvature of the initial data to decay as $x\rightarrow \infty$.
\\ We note that if $g_{0}$ is complete rotationally symmetric with no minimal hyperspheres and $\lvert\text{Rm}_{g_{0}}\rvert_{g_{0}}\rightarrow 0$ at infinity then $\phi(x,0)\rightarrow \infty$; for if $\phi(x,0)\rightarrow r < \infty$ then $\lvert L\rvert \rightarrow 0$ if and only if $\phi_{s}^{2}\rightarrow 1$ which is not possible. Therefore, the decay of the curvature at infinity implies that the radius must be unbounded. In the next Lemma we show that this is enough to control the flow outside a ball uniformly in time. 
\begin{lemma}\label{controlinfinitelimit}
Let $(\R^{n+1},g(t))_{0\leq t < T}$, with $T < \infty$, be the maximal Ricci flow solution evolving from a complete rotationally symmetric metric $g_{0}$ with no minimal embedded hyperspheres and curvature decaying at infinity. Then for any $\epsilon > 0$ there exist $\rho = \rho(\epsilon) > 0$ and $C = C(\epsilon)$ such that 
\[
\sup_{(\R^{n+1}\setminus B(\origin,\rho))\times [0,T)}\lvert \emph{Rm}_{g(t)}\rvert_{g(t)}\leq \epsilon, \,\,\,\,\,\,\,\, \sup_{B(\origin,\rho))\times [0,T)}\phi^{2}\lvert \emph{Rm}_{g(t)}\rvert_{g(t)}\leq C.
\]
\end{lemma}
\begin{proof}
We first note that by rotational symmetry any closed geodesic must lie on a minimal hypersphere; by Lemma \ref{nominimalspheres} we conclude that there are no closed geodesics along the Ricci flow. Since the curvature of $g_{0}$ is bounded, we derive that there exists $\iota = \iota(g_{0}) > 0$ such that $\text{inj}(g_{0}) \geq \iota > 0$. Therefore from \cite{pseudolocalityapplication} we get that for any $\epsilon > 0$ there exists a radius $\rho$ sufficiently large such that the curvature stays bounded by $\epsilon$ in the complement of the Euclidean ball $B(\origin,\rho)$ uniformly in $[0,T)$. The second estimate in the statement follows from Lemma \ref{crucialestimateinterior} once we know that the curvature and hence the radius are uniformly bounded along the hypersphere of radius $2\rho$.
\end{proof}
\begin{remark}
In Lemma \ref{controlinfinitelimit} we consider a much larger set of initial data than that analysed in \cite{woolgar}. In our setting we only require the curvature to decay at spatial infinity at some rate, without prescribing it to be stronger than quadratic. As a consequence of that Lemma \ref{controlinfinitelimit}, for example, applies to initial metrics that open up to infinity either logarithmically or polynomially. 
\end{remark}
\section{Blow-up of Ricci flow with no minimal hyperspheres}
Throughout this section we consider the maximal Ricci flow solution $(\R^{n+1},g(t))$ evolving from a rotationally symmetric metric $g_{0}$ satisfying either the assumptions in Lemma \ref{controlfinitelimit} or those in Lemma \ref{controlinfinitelimit}; moreover, we assume that the maximal time of existence $T$ is finite. By analysing the possible singularity models of the flow we prove Theorem \ref{maintheoremnospheres} and Theorem \ref{maintheoremnospheresimmortal}.
\subsection{Singularity models of Ricci flows with no minimal hyperspheres.}
As observed above, the lack of minimal hyperspheres implies that $(\R^{n+1},g_{0})$ does not contain closed geodesics; therefore the injectivity radius of $g_{0}$ is bounded away from zero. We can thus apply the adaptation of Perelman's no local collapsing theorem \cite{pseudolocality} to complete bounded curvature Ricci flows \cite[Theorem 8.26]{ricciflowtechniques1}: the solution $g(t)$ is weakly $\kappa-$non collapsed in $\R^{n+1}\times (T/2,T)$ at any scale $r\in (0,\sqrt{T/2})$, with $\kappa$ some positive constant only depending on $g_{0}$ and $T$. By \cite{compactnesstheorem} there exist blow-up sequences $(p_{j},t_{j})$ such that $\lambda_{j}\doteq \lvert \text{Rm}_{g(t_{j})}\rvert_{g(t_{j})}(p_{j})\rightarrow \infty$ and the rescaled Ricci flows $(\R^{n+1},g_{j}(t),p_{j})$ defined by $g_{j}(t)\doteq \lambda_{j}g(t_{j}+t/\lambda_{j})$ converge in the pointed Cheeger-Gromov sense to an ancient solution $(M_{\infty},g_{\infty}(t),p_{\infty})_{-\infty < t\leq \omega}$, with $\omega \geq 0$, satisfying
\begin{itemize}
\item[(i)] $g_{\infty}(t)$ is complete,
\item[(ii)] $\sup_{M_{\infty}\times (-\infty,\omega]}\lvert\text{Rm}_{g_{\infty}(t)}\rvert_{g_{\infty}(t)} < \infty$,
\item[(iii)] $g_{\infty}(t)$ is non-flat,
\item[(iii)] $g_{\infty}(t)$ is (weakly) $\kappa$-non collapsed.
\end{itemize} 
\noindent We call any limit ancient solution $(M_{\infty},g_{\infty}(t),p_{\infty})$ a \emph{singularity model} for the flow.
\\By spherical symmetry given a blow-up sequence $(p_{j},t_{j})$ we take $p_{j} = (x_{j},\theta)$ for some $\theta\in S^{n}$; furthermore, without loss of generality we can set $x_{j}< \rho$ whenever there exists $\rho$ defined as in Lemma \ref{controlinfinitelimit}, otherwise the curvature would stay bounded along the blow-up sequence. 
\\In the following we let $\{U_{j}\}$ and $\Phi_{j}: U_{j}\rightarrow B_{g_{j}(0)}(p_{j},2^{j})$ denote the exhaustion of $M_{\infty}$ and the diffeomorphisms realizing the Cheeger-Gromov-Hamilton convergence respectively (see, e.g., \cite[Chapter 4]{ricciflowtechniques1}). 
For any $\nu > 0$ we define
\begin{equation}\label{definitionVj}
V(j,\nu) \doteq \bigcup_{\theta\in S^{n}}B_{g(t_{j})}((x_{j},\theta),\frac{\nu}{\sqrt{\lambda_{j}}}) = \bigcup_{\theta\in S^{n}}B_{g_{j}(0)}((x_{j},\theta),\nu).
\end{equation}
\noindent The rotational symmetry of the solutions ensures that $V(j,\nu)$ are annular regions in $\mathbb{R}^{n+1}$. The sets $V(j,\nu)$ provides a nice exhaustion of the limit manifold $M_{\infty}$; a similar argument for a rotationally invariant K{\"a}hler Ricci flow was discussed in \cite{song}.
\begin{lemma}\label{exhaustionlemma}
Any singularity model is simply connected.
\end{lemma}
\begin{proof}
We first prove a preliminary property.
\begin{claim}\label{claimexistenceuniformnu}
There exists a positive radius $\bar{\nu}$ independent of $j$ such that 
\[
B_{g_{j}(0)}(p_{j},\nu) \subset V(j,\nu)\subset B_{g_{j}(0)}(p_{j},2\nu),
\]
\noindent for any $\nu \geq \bar{\nu}$.
\end{claim}
\begin{proof}[Proof of Claim \ref{claimexistenceuniformnu}] It suffices to show that $d_{g_{j}(0)}((x_{j},\theta),(x_{j},\theta^{\prime}))\leq C < \infty$ , for any $\theta,\theta^{\prime}\in S^{n}$ and uniformly in $j$. The bound follows from Lemma \ref{controlfinitelimit} and Lemma \ref{controlinfinitelimit}; namely, we can find some positive constant $\alpha$ only depending on the dimension such that
\[
d_{g_{j}(0)}((x_{j},\theta),(x_{j},\theta^{\prime})) = \sqrt{\lambda_{j}}\,d_{g(t_{j})}((x_{j},\theta),(x_{j},\theta^{\prime}))\leq \alpha\sqrt{\lambda_{j}}\,\phi(x_{j},t_{j}) \leq C 
\]
\end{proof}
We note that by Claim \ref{claimexistenceuniformnu} the maps $\Phi_{j}^{-1}$ given by Hamilton's Compactness theorem are well defined on $V(j,2^{j-1})$ for $j > j_{0}$, for some $j_{0}$. We can pick $\tilde{j_{0}} > j_{0}$ sufficiently large such that for any $q\in \overline{U_{j_{0}}}$ we have
\[
d_{\Phi_{\tilde{j_{0}} + 2}^{\ast}g_{\tilde{j_{0}}+2}(0)}(p_{\infty},q) \leq 1 + d_{\Phi_{\tilde{j_{0}}}^{\ast}g_{\tilde{j_{0}}}(0)}(p_{\infty},q) \leq 1 + d_{g_{\tilde{j_{0}}}(0)}(p_{\tilde{j_{0}}},\Phi_{\tilde{j_{0}}}(q)) \leq 1 + 2^{\tilde{j_{0}}} < 2^{\tilde{j_{0}}+1}.
\] 
\noindent We thus obtain the following inclusions
\[
\overline{U_{j_{0}}} \subset \Phi_{\tilde{j_{0}}+2}^{-1}\left(B_{g_{\tilde{j_{0}}+2}(0)}(p_{\tilde{j_{0}}+2},2^{\tilde{j_{0}}+1})\right) \subset \Phi_{\tilde{j_{0}}+2}^{-1}\left(V(\tilde{j_{0}}+2,2^{\tilde{j_{0}}+1} )\right) \doteq \tilde{V_{1}}\subset M_{\infty}.
\]
\noindent Since we can iterate the method by replacing $\overline{U_{j_{0}}}$ with $\overline{U_{\tilde{j_{0}}+2}}$, we may conclude that $M_{\infty}$ admits an exhaustion $\{\tilde{V}_{j}\}$ of simply connected open sets; that completes the proof.  
\end{proof}
Next we characterize the geometry of the possible singularity models.
\begin{lemma}\label{characterizationsingularitymodels}
Either one of the following is satisfied:
\begin{itemize}
\item[(i)] $(M_{\infty},g_{\infty}(t))$ is the self-similar shrinking Ricci soliton on $\R\times S^{n}$.
\item[(ii)]$(M_{\infty},g_{\infty}(t))$ is a conformally flat positively curved $\kappa$-solution, for some $\kappa > 0$. Moreover $M_{\infty} = \R^{n+1}$.  
\end{itemize}
\end{lemma}
\begin{proof}
If $n= 2$ (i.e. the three-dimensional case) then the Cotton tensor of the singularity model is identically zero; moreover, from \cite{stronguniqueness} it follows that the curvature operator of the singularity model is nonnegative. Similarly, when $n\geq 3$ the Weyl tensor of the singularity model is identically zero; accordingly, by \cite{zhang} we also derive that the curvature operator is nonnegative. In particular, since any singularity model is weakly $\kappa$-non collapsed at all scales, we find that any singularity model is a $\kappa$-solution. 
\\Assume that there exists $q\in M_{\infty}$ and $t\in (-\infty,\omega]$ such that the curvature operator is not strictly positive. Then the strong maximum principle for systems implies that the holonomy of the singularity model is not maximal \cite{positivecurvature}. Since $(M_{\infty},g_{\infty}(t))$ is non-compact, non-flat and has nonnegative curvature we find that $g_{\infty}(t)$ either has reducible holonomy or is K\"ahler. However, for any real dimension $2m\geq 4$, any conformally flat K\"ahler manifold is scalar flat \cite[Proposition 2.68]{besse}; therefore, we may conclude that the holonomy is reducible. Since $M_{\infty}$ is simply connected (Lemma \ref{exhaustionlemma}) and $g_{\infty}(t)$ is complete, we may apply de Rham's decomposition theorem and obtain that the singularity splits off a line. Finally a Riemannian product $\R\times N$ is conformally flat if and only if $N$ is a space form \cite[Section C]{incollection}; therefore we find that $(M_{\infty},g_{\infty}(t))$ must be the self-similar shrinking soliton on $\R\times S^{n}$. 
\\Suppose that there exists $q\in M_{\infty}$ and $t\in (-\infty,\omega]$ such that $\text{Rm}_{g_{\infty}(t)}(q) > 0$. By Hamilton's strong maximum principle \cite{positivecurvature} we derive that the curvature operator is positive everywhere in the space-time; from the soul theorem it also follows that $M_{\infty}= \R^{n+1}$.
\end{proof} 
\begin{remark}
It is not clear to us whether one can \emph{a priori} deduce that any singularity model is in fact rotationally symmetric; however, for our purposes it suffices to know that any singularity model which is not a family of shrinking cylinders must be a conformally flat positively curved $\kappa$-solution.
\end{remark}
\begin{remark}\label{remarkbrendle}
We also point out that one can actually refine the previous classification even further. In the three-dimensional case we can apply the recent analysis in \cite{brendle} to derive that any singularity model which is not a family of shrinking cylinders is isometric to the Bryant soliton (up to scaling). Similarly, when $n\geq 3$ by \cite{catino2} we obtain that any conformally flat ancient solution as in (ii) of Lemma \ref{characterizationsingularitymodels} is in fact rotationally symmetric; by the generalization of \cite{brendle} to higher dimensions \cite{brendle2} we may finally conclude that any singularity model with positive curvature operator is isometric to the Bryant soliton.
\end{remark}

\subsection{Type-II singularities.}
We may now address the proof of Theorem \ref{maintheoremnospheres}; the argument relies on the characterization of Type-I singularities described in \cite{type1}. Similarly to \cite{type1}, we introduce the singular set $\Sigma \subset \mathbb{R}^{n+1}$, which is defined by the following property: given a point $p\in\mathbb{R}^{n+1}$ then $\lvert \text{Rm}_{g(t)}\rvert_{g(t)}$ stays bounded in some neighbourhood of $p$ as $t\nearrow T$ if and only if $p\in\mathbb{R}^{n+1}\setminus \Sigma$.
\begin{proof}[Proof of Theorem \ref{maintheoremnospheres}] Assume for a contradiction that $g(t)$ is a Type-I Ricci flow. If the origin is not in $\Sigma$ then by definition of singular set we can find some small $\delta > 0$ such that $\lvert \text{Rm}_{g(t)}\rvert_{g(t)}\leq C < \infty$ uniformly in $B(\origin,2\delta)\times [0,T)$. Therefore there exists $\nu > 0$ such that $\phi(\delta,t)\geq \nu > 0$ for any $t\in [0,T)$; by Lemma \ref{nominimalspheres} we get $\phi(x,t)\geq \phi(\delta,t)\geq \nu$ for any $x\geq \delta$ and for any $t\in[0,T)$. From (ii) of Lemma \ref{controlfinitelimit} we derive that $\lvert\text{Rm}\rvert \leq C$ outside $B(\origin,2\delta)$ uniformly in time; by \cite{shi} the last condition implies that the flow smoothly extends to time $T$, which is a contradiction. 
\\We may thus consider the case $\origin\in\Sigma$. By \cite[Theorem 1.1]{type1} we can parabolically dilate the solution at the origin and obtain, up to passing to a subsequence, a \emph{non-flat} gradient shrinking soliton in canonical form $(M_{\infty},g_{\infty}(t))$. Since the soliton is simply connected (Lemma \ref{exhaustionlemma}) and conformally flat (Lemma \ref{characterizationsingularitymodels}), from the classification in \cite[Theorem 1.2]{zhang} we derive that 
$(M_{\infty},g_{\infty}(t))$ must be a shrinking cylinder. By the Cheeger-Gromov-Hamilton convergence and the rotational symmetry we conclude that the cylinder $\R\times S^{n}$ is exhausted by open sets diffeomorphic to $\R^{n+1}$; that gives a contradiction\footnote{Explicitly, if the cylinder admitted an exhaustion by open sets diffeomorphic to $\R^{n+1}$, then its rank 1 compactly supported de Rham cohomology group would be trivial.}. Since by (i) of Lemma \ref{controlfinitelimit} $g(t)$ develops a finite-time singularity, we have just shown that this singularity must be Type-II.
\\ Once we know that $g(t)$ is a Type-II flow we can pick a blow-up sequence whose associated singularity model $(M_{\infty},g_{\infty}(t))$ is an \emph{eternal} solution to the Ricci flow with $\lvert \text{Rm}_{\infty}\rvert$ attaining its supremum in the space-time (see, e.g., \cite[Section 16]{formationsingularities}). By Lemma \ref{characterizationsingularitymodels} we derive that $M_{\infty} = \R^{n+1}$ and that $g_{\infty}(t)$ is a $\kappa$-solution with positive curvature operator, because shrinking cylinders extinguish in finite time. Up to choosing a different blow-up sequence we may thus assume that there exists a singularity model which is a conformally flat eternal solution to the Ricci flow with bounded positive curvature operator and scalar curvature attaining its supremum in the space-time; by Hamilton's rigidity result \cite{eternal} we deduce that $(M_{\infty},g_{\infty}(t))$ is a steady gradient Ricci soliton. Finally by \cite{cao} we conclude that this steady soliton is isometric to the Bryant soliton (up to scaling) \cite{bryant}.
\end{proof}
\begin{remark}\label{remarkbrendle2}
According to Remark \ref{remarkbrendle}, one can improve the result in Theorem \ref{maintheoremnospheres} and obtain that due to \cite{brendle} in dimension three and \cite{catino2} and \cite{brendle2} in dimension greater than three any blow-up sequence whose associated singularity model is not a shrinking cylinder gives rise to the Bryant soliton in the limit.
\end{remark}
\begin{remark}
The argument above highlights that the lack of minimal spheres guarantees that the curvature does not concentrate locally around some neck-region; the singularity is hence \emph{slowly} forming. We also point out that the non-compactness of the underlying manifold played a crucial role; in the analogous case of $S^{n+1}$ one has to take into account global Type-I singularities where the volume of the manifold approaches zero in the limit.
\end{remark}
We finally show that the Bryant soliton has to appear at the origin $\origin\in\R^{n+1}$ if the pinching condition \eqref{pinchingcondition} is satisfied.
\begin{corollary}\label{corollarypinching}
Under the same hypotheses as Theorem \ref{maintheoremnospheres}, if further the initial metric satisfies $K_{g_{0}}\leq L_{g_{0}}$ then there exists a sequence $t_{j}\nearrow T$ such that the rescaled Ricci flows $(\R^{n+1},g_{j}(t),\origin)$ defined by $g_{j}(t) = R_{g(t_{j})}(\origin)g(t_{j}+t/R_{g(t_{j})}(\origin))$ on $[-R_{g(t_{j})}(\origin)t_{j},0]$ converge to the Bryant soliton (up to scaling).
\end{corollary}
\begin{proof} By Theorem \ref{maintheoremnospheres} we know that the flow develops a Type-II singularity at some $T<\infty$ and that there exist rescaled Ricci flows $(\R^{n+1},g_{j}(t),p_{j})$ defined by $g_{j}(t)=R_{g(t_{j})}(p_{j})g(t_{j}+t/R_{g(t_{j})}(p_{j}))$ smoothly converging to the Bryant soliton, for some $t_{j}\nearrow T$. We can then apply Lemma \ref{tipregion} and conclude that the rescaled sequence $(\R^{n+1},g_{j}(t),\origin)$ defined by $g_{j}(t)=R_{g(t_{j})}(\origin)g(t_{j}+t/R_{g(t_{j})}(\origin))$ converges to the Bryant soliton as well.
\end{proof}
\subsection{Immortal solutions.} We now analyse the maximal Ricci flow solution $g(t)$ evolving from a rotationally symmetric metric $g_{0}$ with no minimal embedded hyperspheres and curvature decaying to zero at infinity. We assume that such solution develops a singularity at some $T < \infty$ and we aim to exhibit a contradiction, hence proving Theorem \ref{maintheoremnospheresimmortal}. We first show that $\phi_{s}$ admits a uniform positive lower bound in the compact region where singularities may form, then we use the exhaustion constructed in Lemma \ref{exhaustionlemma} to prove that the condition about $\phi_{s}$ implies that the singularity model $(M_{\infty},g_{\infty}(t))$ has positive asymptotic volume ratio. 
\begin{proof}[Proof of Theorem \ref{maintheoremnospheresimmortal}]
By Lemma \ref{controlinfinitelimit} we deduce that there exists $\rho_{1} > 0$ such that $\lvert\text{Rm}_{g(t)}\rvert_{g(t)}\leq C$ on $\R^{n+1}\setminus B(\origin,\rho_{1})$ uniformly with respect to time, for some $C > 0$; furthermore, the estimate \eqref{crucialestimate} holds in $B(\origin,\rho_{1})\times [0,T)$. In particular, from Lemma \ref{existlimit} and Lemma \ref{nominimalspheres} it follows that there exists $\rho_{0} < \rho_{1}$ satisfying $\lim_{t\nearrow T}\phi(x,t) = 0$ for any $x < \rho_{0}$ while $\lim_{t\nearrow T}\phi(x,t) > 0$ for all $x\in (\rho_{0},\rho_{1}]$. 
\\Given any $x\in(\rho_{0},\rho_{1})$ by Lemma \ref{controlinfinitelimit} we find that 
\[
(\lvert 1-\phi_{s}^{2}\rvert + \lvert\phi_{ss}\rvert)(x,t) = (\phi^{2}\lvert L\rvert + \phi\lvert K \rvert)(x,t)\leq C(x) < \infty,
\]
\noindent for all $t\in[0,T)$. Therefore both $\phi_{s}$ and $\phi_{ss}$ are uniformly bonded at any radius $x\in(\rho_{0},\rho_{1})$. Once we choose $\rho_{1}$ large enough, we let $\rho_{0} + 1 < x_{0} < x_{1} < \rho_{1}$ and $\delta > 0$ satisfy
\[
\phi(x_{1},t) - \phi(x_{0},t) \geq \phi(x_{1},0)e^{-CT} - \phi(x_{0},0)e^{CT} \geq \delta > 0,
\]
\noindent where we have used that given $\delta > 0$ we can always find $x_{1}$ and $x_{0}$ as above because $\phi(x,0)\rightarrow \infty$ as observed in Section 3.2.
\begin{claim}\label{claim1}
There exists $\tilde{x} > \rho_{0}$ and $\tilde{\beta} > 0$ such that $\phi_{s}(\tilde{x},t)\geq \tilde{\beta} > 0$ for any $t\in[0,T)$.
\end{claim}
\begin{proof}[Proof of Claim \ref{claim1}] 
Let $x > \rho_{0}$; once we fix an angle $\theta\in S^{n}$ we can extend a local $\hat{g}$-orthonormal frame $\{e_{i}\}$ on $S^{n}$ to a $g(t)$-orthonormal frame around $p = (x,\theta)$ of the form $\{\partial_{s},e_{i}/\phi\}$ for any $t\in [0,T)$. Using the commutator formula \eqref{commutatorformula} and the Koszul formula we find
\begin{align}
\partial_{t}(\partial_{s}\phi)(p,t) &= \partial_{s}\left(-\text{Ric}_{g(t)}\left(\frac{e_{i}}{\phi},\frac{e_{i}}{\phi}\right)\phi\right)(p,t) + nK\phi_{s}(p,t) \notag \\ &= -\text{Ric}_{g(t)}\left(\frac{e_{i}}{\phi},\frac{e_{i}}{\phi}\right)\phi_{s}(p,t) - \left(\text{Ric}_{g(t)}\left(\frac{e_{i}}{\phi},\frac{e_{i}}{\phi}\right)\right)_{s}\phi(p,t) + nK\phi_{s}(p,t) \notag \\ &= -\text{Ric}_{g(t)}\left(\frac{e_{i}}{\phi},\frac{e_{i}}{\phi}\right)\phi_{s}(p,t) -\nabla_{g(t)}\text{Ric}_{g(t)}\left(\partial_{s},\frac{e_{i}}{\phi},\frac{e_{i}}{\phi}\right)\phi(p,t) + nK\phi_{s}(p,t). \label{firstderivativelipschitz}
\end{align}
\noindent Since $x > \rho_{0}$ there exists $\gamma = \gamma(x) > 0$ such that $\phi(x,t)\geq \gamma > 0$ as long as the solution exists. From Lemma \ref{controlinfinitelimit} and Shi's derivative estimates we deduce that $(\lvert\text{Rm}_{g(t)}\rvert + \lvert\nabla_{g(t)}\text{Rm}_{g(t)}\rvert)(x,t) \leq C(x) < \infty$ uniformly in $[0,T)$. Therefore we can bound the right hand side of \eqref{firstderivativelipschitz} by a uniform positive constant only depending on $x$, thus obtaining that $\phi_{s}(x,\cdot)$ is a Lipschitz function of time on $[0,T)$. A consequence of this fact is that $\phi_{s}(x,\cdot)$ admits a (finite) limit as $t\nearrow T$ for any $x > \rho_{0}$. 
\\Assume for a contradiction that any such limit is zero. Since the curvature is controlled in the annular region $(x_{0},x_{1})\times S^{n}$, by standard distortion estimates of the distance we get
\begin{align*}
\delta &\leq \phi(x_{1},t) - \phi(x_{0},t) \leq \sup_{[x_{0},x_{1}]}\phi_{s}(\cdot,t)(s(x_{1},t)-s(x_{0},t))\\ &\leq C\,\sup_{[x_{0},x_{1}]}\phi_{s}(\cdot,t)(s(x_{1},0)-s(x_{0},0))\leq C\,\sup_{[x_{0},x_{1}]}\phi_{s}(\cdot,t),
\end{align*}
\noindent for any $t\in[0,T)$. Since we have seen that $\phi_{ss}$ is uniformly bounded in $[x_{0},x_{1}]\times[0,T)$ we conclude that $\sup_{[x_{0},x_{1}]}\phi_{s}(\cdot,t)\rightarrow 0$ as $t\nearrow T$ as long as $\phi_{s}(x,t)\rightarrow 0$ for any $x > \rho_{0}$; that is a contradiction. Thus there exists $\tilde{x} > \rho_{0}$ such that $\lim_{t\nearrow T}\phi_{s}(\tilde{x},t) \neq 0$; by Lemma \ref{nominimalspheres} we deduce that there exists $\beta > 0$ as in the statement of the Claim. 
\end{proof}
From the boundary conditions \eqref{boundaryconditions} and Claim \ref{claim1} we derive that if $\phi_{s}$ approaches zero in $B(\origin,\tilde{x})$ as $t\nearrow T$, then this must happen along a sequence of interior minima; however, the maximum principle applied to the evolution equation \eqref{equationfirstderivative} shows that this is not possible. We thus find $\beta > 0$ such that for any $t\in [0,T)$ the following holds:
\begin{equation}\label{lowerboundphis}
\inf_{B(\origin,\tilde{x})}\phi_{s}(\cdot,t) \geq \beta > 0.
\end{equation}
\noindent Consider a standard parabolic rescaling $g_{j}(t)$ along a blow-up sequence $(p_{j},t_{j})$ converging smoothly on compact sets to a singularity model $(M_{\infty},g_{\infty}(t),p_{\infty})_{t\in (-\infty,\omega]}$. 
\noindent We need to show that the rescaled geodesic balls stay inside $B(\origin,\tilde{x})$ for $j$ large enough.
\begin{claim}\label{claiminclusionballs}
For any $\nu >0$ there exists $j_{0} = j_{0}(\nu)$ such that for all $j\geq j_{0}$ the following holds:
\[
B_{g_{j}(0)}(p_{j},\nu) \subset B(\origin,\tilde{x}).
\]
\end{claim}
\begin{proof}[Proof of Claim \ref{claiminclusionballs}] Assume for a contradiction that there exist $\nu > 0$ and a subsequence $q_{j} = (y_{j},\theta_{j})$ such that $q_{j}\in B_{g_{j}(0)}(p_{j},\nu)$ and $y_{j} > \tilde{x}$. By the Cheeger-Gromov-Hamilton convergence $\Phi_{j}^{-1}(q_{j})\in B_{g_{\infty}(0)}(p_{\infty},2\nu)$ for $j$ large enough. Therefore by Lemma \ref{controlinfinitelimit} we find that $R_{g_{\infty}(0)}$ vanishes at some $q_{\infty}\in B_{g_{\infty}(0)}(p_{\infty},2\nu)$; since the singularity model has nonnegative curvature operator (Lemma \ref{characterizationsingularitymodels}) by a standard application of the maximum principle we deduce that $g_{\infty}(0)$ is flat.
\end{proof}
Let $\nu > \bar{\nu}$ with $\bar{\nu}$ given in Claim \ref{claimexistenceuniformnu}. Consider the annular region $V(j,\nu)$ defined as in \eqref{definitionVj}; we let $\mu_{j} > 0$ be the positive quantity satisfying 
\[
s(x_{j}+\mu_{j},t_{j})-s(x_{j},t_{j}) \equiv \int_{x_{j}}^{x_{j}+\mu_{j}}\xi(x,t_{j})dx = \frac{\nu}{\sqrt{\lambda_{j}}}.
\]
\noindent Assume that $j$ is large enough such that the inclusion in Claim \ref{claiminclusionballs} is verified; equivalently, we have $x_{j} + \mu_{j} \leq \tilde{x}$. From Lemma \ref{nominimalspheres} and the lower bound \eqref{lowerboundphis} it follows that
\[
\phi(x,t_{j}) = \int_{0}^{x}\phi_{x}(y,t_{j})dy \geq \int_{x_{j}}^{x}(\phi_{s}\xi)(y,t_{j})dy \geq \beta (s(x,t_{j}) - s(x_{j},t_{j})),
\]
\noindent for any $x\in (x_{j},x_{j}+\mu_{j})$. We thus obtain 

\begin{align}
\text{Vol}_{g(t_{j})}(V(j,\nu)) & \geq \text{Vol}_{g(t_{j})}(S^{n}\times (x_{j},x_{j}+\mu_{j})) = C(n)\int_{x_{j}}^{x_{j}+\mu_{j}}(\phi^{n}\xi)(x,t_{j})dx \notag \\ &\geq C(n)\int_{x_{j}}^{x_{j}+\mu_{j}}\beta^{n}(s(x,t_{j}) - s(x_{j},t_{j}))^{n}\xi(x,t_{j})dx \notag \\ &= C(n)(s(x_{j}+\mu_{j},t_{j}) - s(x_{j},t_{j}))^{n+1} = C(n)\frac{\nu^{n+1}}{\lambda_{j}^{\frac{n+1}{2}}}, \label{lastalignforglobal}
\end{align}
\noindent for some positive constant $C$ independent of $j$ that we have renamed from line to line.
\\ We finally conclude that for any $\nu \geq \bar{\nu}$, with $\bar{\nu}$ defined in Claim \ref{claimexistenceuniformnu}, by the Cheeger-Gromov-Hamilton convergence and \eqref{lastalignforglobal} there exists $j$ large enough satisfying 
\[
\text{Vol}_{g_{\infty}(0)}B_{g_{\infty}(0)}(p_{\infty},4\nu) \geq \text{Vol}_{g_{j}(0)}B_{g_{j}(0)}(p_{j},2\nu)\geq \text{Vol}_{g_{j}(0)}V(j,\nu) \geq C(n)\nu^{n+1}.
\]
\noindent The last inequality implies that $g_{\infty}(0)$ has Euclidean volume growth: equivalently, the asymptotic volume ratio of $g_{\infty}(0)$ is positive. Since we have already shown that any singularity model is a $\kappa$-solution (see Lemma \ref{characterizationsingularitymodels}), by \cite[Proposition 11.4]{pseudolocality} we derive the contradiction; thus the solution is immortal.
\end{proof}
\section{Rotationally symmetric Ricci flow with nonnegative curvature}
In this section we classify rotationally symmetric Ricci flows with bounded nonnegative Ricci tensor; we also comment on some properties satisfied by standard solutions in the sense of Lu and Tian.
\\Let $g_{0}$ be a complete rotationally symmetric metric on $\R^{n+1}$ with bounded nonnegative Ricci curvature and consider the maximal Ricci flow solution $(\R^{n+1},g(t))_{0\leq t < T}$ evolving from $g_{0}$. Since by rotational symmetry the Ricci tensor is nonnegative if and only if the curvature operator is nonnegative, we can apply Hamilton's strong maximum principle \cite{positivecurvature} and conclude that $\text{Ric}_{g(t)} > 0$ for all $t\in (0,T)$ because the curvature operator of $g_{0}$ has non trivial kernel at any point if and only if $g_{0}$ is flat. We first report the following result, which allows to compute the maximal time of existence for positively curved solutions that are asymptotic to a round cylinder at spatial infinity.
\begin{proposition}[Chen and Zhu \cite{chenzhu}, Theorem A.1]\label{propchenzhu} Let $(\R^{n+1},g(t))_{0\leq t < T}$ be the maximal Ricci flow solution evolving from a complete metric $g_{0}$ with bounded nonnegative curvature operator, positive scalar curvature and which is asymptotic to a round cylinder of radius $r_{0}$ at spatial infinity. Then the solution satisfies: 
\[ 
\inf_{p\in\R^{n+1}}R_{g(t)}(p)\geq \frac{C}{T-t},
\]
\noindent for some $C > 0$, where $T = \frac{r_{0}^{2}}{2(n-1)}$.
\end{proposition}
We may now address the proof of Corollary \ref{classificationpositive}.  
\begin{proof}[Proof of Corollary \ref{classificationpositive}]
Given $ t > 0$ we have $K(\cdot,t) > 0$, which is equivalent to $\phi_{ss}(\cdot,t) < 0$ on $\R^{n+1}\setminus\{\origin\}$; the last condition implies $0 < \phi_{s}(\cdot,t) \leq 1$ because $g(t)$ is a complete metric\footnote{We note that the strict inequality $\phi_{s}(\cdot,t) > 0$ for any $t\in (0,T)$ follows from the boundary conditions \eqref{boundaryconditions} and the real-analyticity of the Ricci flow \cite{bando}.}. Therefore $(\R^{n+1},g(t))$ does not contain minimal embedded hyperspheres for any $t\in(0,T)$. 
We then need to consider two different cases, depending on whether the radius $\phi(\cdot,t)$ is bounded or unbounded.
\\ \emph{The radius is bounded.} Suppose that $\phi(x,0)\rightarrow r_{0} < \infty$. The same argument for the proof of (ii) in Lemma \ref{controlfinitelimit} shows that the solution is smoothly asymptotic to a round cylinder at spatial infinity for any $t\in (0,T)$. From Proposition \ref{propchenzhu} we derive that the solution develops a global singularity at $T = r_{0}^{2}/2(n-1)$. We finally apply Theorem \ref{maintheoremnospheres} to deduce that the singularity is Type-II and is modelled on the Bryant soliton once suitably dilated. 
\\ \emph{The radius is unbounded.} If $\phi(x,0)\rightarrow \infty$ as $x\rightarrow \infty$ then $\phi(x,t)\rightarrow \infty$ for any $t\in [0,T)$ because the curvature stays bounded until time $T$ by \cite{shi}. Let us fix $t_{0} > 0$; since $0 < \phi_{s}(\cdot,t_{0})\leq 1$ we get $\lvert L(\cdot,t_{0})\rvert \rightarrow 0$ as $x\rightarrow\infty$. Furthermore $\phi_{s}(\cdot,t_{0})$ has a (finite) limit at infinity being $\phi_{ss}(\cdot,t_{0})\leq 0$; it follows that $\phi_{ss}(\cdot,t_{0})$ and hence $K(\cdot,t_{0})$ are integrable. Since by Shi's derivative estimates $\lvert K_{s}(\cdot,t_{0})\rvert\leq C$ we obtain $K(x,t_{0})\rightarrow 0$ as $x\rightarrow \infty$. We have thus shown that $\lvert \text{Rm}_{g(t_{0})}\rvert_{g(t_{0})} \rightarrow 0$ at infinity; by applying Theorem \ref{maintheoremnospheresimmortal} we conclude that the solution is immortal. 
\end{proof}
\subsection{Standard solutions to the Ricci flow.} 
In the following we relate the classification in Corollary \ref{classificationpositive} to the family of standard solutions introduced by Lu and Tian in \cite{peng}, where they generalized Perelman's class of special solutions discussed in \cite{perelmansurgery}. 
According to \cite{peng}, we have the following:
\begin{definition}\label{definitioninitialstandard}
We let $\G$ be the set of (smooth) complete rotationally symmetric metrics $g$ on $\R^{n+1}$ satisfying
\begin{itemize}
\item[(i)] There exists a sequence of points $p_{j}\rightarrow \infty$ in $\R^{n+1}$ such that $(\R^{n+1},g,p_{j})$ converges to the round cylinder $(\R \times S^{n}, (dx)^{2} + r^{2}\hat{g},p^{\ast})$ in pointed $C^{3}$ Cheeger-Gromov topology, for some $r > 0$.
\item[(ii)] $\text{Rm}_{g} \geq 0$ everywhere and $\text{Rm}_{g}(p) > 0$ for some $p\in\R^{n+1}$.
\item[(iii)] There exists $\alpha > 0$ such that on $\R^{n+1}$ 
\[
\lvert\text{Rm}_{g}\rvert_{g} + \sum_{k = 1}^{4}\lvert \nabla^{k} \text{Rm}_{g}\rvert_{g} \leq \alpha.
\]
\end{itemize}
\end{definition}
\begin{remark}\label{remarkarctan} To prove that $\G$ is non-empty and to get a feeling for the metrics contained in this set, we consider $g$ of the form \eqref{rotationallysymmetrictime0} with $\phi(s) = \arctan(s)$. We first verify that $\phi$ satisfies \eqref{boundaryconditions} so that $g$ is smooth, complete, and with bounded curvature on any compact region. Conditions (ii) and (iii) of Definition \ref{definitioninitialstandard} follow from the formulas for the sectional curvatures \eqref{formulasforKandL}. For what concerns the convergence to the round cylinder of radius $r = \pi/2$, we introduce the exhaustion $U_{j} \doteq (-j,\infty)\times S^{n}$ and the family of translations $\Phi_{j}:U_{j}\rightarrow (j,\infty)\times S^{n}\doteq V_{j}$ defined by $s\mapsto s + 2j$. Once we fix an angle $\theta\in S^{n}$, the embeddings satisfy $\Phi_{j}(0,\theta) = (2j,\theta)$ for any $j$. If we denote the points $(2j,\theta)$ by $p_{j}$ and $(0,\theta)$ by $p^{\ast}$, property (i) in Definition \ref{definitioninitialstandard} is then equivalent to showing that 
\[g|_{V_{j}} \xrightarrow{C^{3}}  g_{\text{cyl}}, \,\,\,\,\,\,\,\,\,\ g_{\text{cyl}} = (dx)^{2} + (\pi^{2}/4)\hat{g}.\]
\noindent Finally such convergence follows from the fact that $\lvert \partial_{s}^{k}\text{arctan}(s)\rvert\rightarrow 0$ uniformly in $V_{j}$ as $j\rightarrow \infty$, for any $k = 1,2,3$.
\end{remark}
Since any $g_{0}\in\G$ is complete with bounded curvature there exists a solution to the Ricci flow starting at $g_{0}$ \cite{shi}; moreover, such solution is unique in the class of complete solutions with bounded curvature on compact subintervals \cite{uniqueness}. As a consequence of that, we may give the following definition, due to Lu and Tian \cite{peng}.
\begin{definition}\label{standardsolutions}
Let $g_{0}\in\G$. The maximal Ricci flow solution starting at $g_{0}$ is called a \emph{standard solution}.
\end{definition}
We first report a result by Lu and Tian which shows that $\G$ is closed with respect to the Ricci flow problem. In the following we let $p_{j}$ and $r$ be as in Definition \ref{definitioninitialstandard}.
\begin{lemma}[Lu and Tian \cite{peng}, Lemma 1]\label{penglemma}
Let $(\R^{n+1},g(t))_{0\leq t < T}$ be the maximal solution to the Ricci flow starting at some $g_{0}\in\G$. For any $T^{\prime}\in (0,T)$ there exists a subsequence of $(\R^{n+1},g(t),p_{j})_{0\leq t \leq T^{\prime}}$ which converges in $C^{3}$ Cheeger-Gromov pointed topology to a self-similar shrinking cylinder
\[
g_{\emph{cyl}}(t) = (ds)^{2} + (r^{2} - 2(n-1)t)\hat{g}.
\]
\end{lemma}
The previous Lemma, Shi's derivative estimates \cite{shi} and Hamilton's strong maximum principle for systems \cite{positivecurvature} imply the following
\begin{corollary}\label{standardsolutionspreserved}
Let $(\R^{n+1},g(t))_{0\leq t < T}$ be the maximal solution to the Ricci flow starting at some $g_{0}\in\G$. Then $g(t)\in\G$ for any $t\in [0,T)$.
\end{corollary} 
By Proposition \ref{propchenzhu} we know that standard solutions survive until some finite time $T$ which only depends on the dimension and the asymptotic round cylinder at infinity; at time $t = T$ the solution extinguishes globally. In order to fully understand the nature of this singularity one needs to classify its type and the possible limits of blow-ups; in this regard, item (i) of Corollary \ref{classificationpositive} immediately gives us the following, which also proves the conjecture by Chow and Tian in Corollary \ref{conjecture}.
\begin{theorem}\label{theoremconjecture}
Any standard solution in the sense of Lu and Tian develops a global Type-II singularity at some finite time $T<\infty$; moreover, the singularity is modelled on the Bryant soliton once suitably dilated.
\end{theorem}
 According to \cite{peng} in the previous result we do not require standard solutions to have a well-defined tip. However we can provide a simple characterization of those standard solutions whose curvature concentrates at the origin; the conclusions of Corollary \ref{corollarypinching} can be strengthened in the case of standard solutions by using the trace of the Harnack estimate \cite{harnack} and the results in \cite{brendle} and \cite{brendle2}.
\begin{proof}[Proof of Corollary \ref{corollarybrendlino}]
Given a sequence $t_{j}\nearrow T$, since the solution has nonnegative curvature and hence the flow is only controlled by the scalar curvature, by Lemma \ref{tipregion} and the trace of the Harnack estimate we conclude that the rescaled Ricci flows $(\mathbb{R}^{n+1},g_{j}(t),\origin)$ defined on $[-R_{g(t_{j})}(\origin)t_{j},0]$ by $g_{j}(t)\doteq R_{g(t_{j})}(\origin)g(t_{j}+t/R_{g(t_{j})}(\origin))$ (sub)converge in the pointed Cheeger-Gromov topology. Lemma \ref{characterizationsingularitymodels} then implies that any limit must be a conformally flat $\kappa$-solution with positive curvature, for otherwise the cylinder would be exhausted by open sets diffeomorphic to $\R^{n+1}$. We can apply \cite{brendle} when $n = 2$ and \cite{catino2} and \cite{brendle2} when $n > 2$ as explained in Remark \ref{remarkbrendle}.
\end{proof}
\section{Ricci flow with necks}
In this section we adapt the analysis in \cite{exampleneck} to address the expectations in \cite{woolgar} discussed in the introduction; accordingly, we consider rotationally invariant asymptotically flat Ricci flows containing minimal hyperspheres. We explicitly describe a characterization of a \emph{sufficiently pinched} initial neck leading to the formation of a Type-I singularity and we provide examples of initial necks disappearing in finite time along the Ricci flow. For simplicity, in the following we discuss the case where the solution has only one neck, but the conclusions easily generalize to the case of multiple necks. 
\subsection{Type-I neckpinches.} 
Throughout this subsection we consider the maximal Ricci flow solution $g(t)$ defined on $[0,T)$, for some $T < \infty$, evolving from a rotationally symmetric asymptotically flat metric $g_{0}$; we refer to \cite{woolgar} for a detailed analysis of this condition along the flow. For our purposes it suffices to note that there exists $\epsilon > 0$ such that 
\begin{equation}\label{asymptoticflat}
x^{2+\epsilon}\lvert\text{Rm}_{g(t)}\rvert_{g(t)}\leq C(t) < \infty
\end{equation}
\noindent for some positive constant depending continuously on $t\in[0,T)$. In the following we call such flow an $\epsilon$-\emph{asymptotically flat} Ricci flow.
\begin{lemma}\label{sturmiantheoremneckpin}
Let $(\R^{n+1},g(t))_{0\leq t < T}$ be a rotationally symmetric $\epsilon$-asymptotically flat Ricci flow, for some $\epsilon > 0$ and $T < \infty$. Then there exist $\rho > 0$ and $C > 0$ such that
\[
\sup_{(\R^{n+1}\setminus B(\origin,\rho))\times [0,T)}\phi^{2 + \frac{\epsilon}{2}}\lvert\emph{Rm}_{g(t)}\rvert_{g(t)} \leq C.
\]
\end{lemma}
\begin{proof}
Since the curvature is decaying to zero at infinity and the metric is close to the Euclidean metric, in a precise way, outside a compact region, we can apply \cite{pseudolocalityapplication} and deduce that there exists $\rho > 0$ such that $\lvert\text{Rm}_{g(t)}\rvert_{g(t)}\leq 1$ in $\R^{n+1}\setminus B(\origin,\rho)$ uniformly in time. Define $f\doteq \phi^{\alpha}\lvert\text{Rm}_{g(t)}\rvert_{g(t)}^{2}$ on the complement of $B(\origin,\rho)$ with $\alpha = 4 + \epsilon$. From \eqref{asymptoticflat} it easily follows that given $t\in [0,T)$ we have $\phi(x,t)\sim x$ for $x$ large enough \cite{woolgar}; thus we derive that $f$ is uniformly bounded along the parabolic boundary of the region. It then suffices to show that $f$ cannot diverge along a sequence of interior maxima; the evolution equation of $f$ is given by
\[
f_{t} \leq \Delta f - 4\alpha\phi^{\alpha -1}\phi_{s}\lvert\text{Rm}\rvert(\lvert\text{Rm}\rvert)_{s} + \lvert\text{Rm}\rvert^{2}\phi^{\alpha}\left(\frac{\alpha}{\phi^{2}}(-(n-1) - \alpha\phi_{s}^{2}) + C\lvert\text{Rm}\rvert \right)
\]
where we have used a standard estimate for the evolution of the curvature along the Ricci flow. At any interior maximum point $(p_{0},t_{0})$ we find 
\[
f_{t}(p_{0},t_{0}) \leq \lvert\text{Rm}\rvert^{2}\phi^{\alpha}\left(\frac{\alpha}{\phi^{2}}(-(n-1) + \alpha\phi_{s}^{2}) + C\lvert\text{Rm}\rvert \right)(p_{0},t_{0}).
\]
\noindent Therefore, as long as $T < \infty$ we find
\[
\dot{f}_{\text{max}} \leq C f_{\text{max}}
\]
\noindent where the inequality follows again from the curvature being uniformly bounded in the region\footnote{Explicitly, once we know that $\phi_{s}\rightarrow 1$ at spatial infinity as long as the solution exists (\cite{woolgar}), we can then apply the maximum principle to the evolution equation of $\phi_{s}$ as in Lemma \ref{crucialestimateinterior}.}. By integrating we obtain that $f$ is uniformly bounded as long as $T < \infty$.
\end{proof}
Let $g_{0}$ be rotationally symmetric and $\epsilon$-asymptotically flat. Adopting the same notations as in \cite{exampleneck}, we call a local maximum(minimum) for $\phi(\cdot,0)$
a bump(neck) for $g_{0}$; in the case of the minimum we never consider the origin. We say that a bump(neck) is degenerate when the spatial second derivative of $\phi$ vanishes at the maximum(minimum) point. Otherwise, the bump(neck) is referred to as nondegenerate. We note that from the boundary conditions \eqref{boundaryconditions} and the asymptotics \eqref{asymptoticflat} the existence of  a neck always implies the existence of a bump and vice versa. In the case of a single bump(neck) we are dealing with, one can use Lemma \ref{sturmiantheoremneckpin} to generalize \cite[Lemma 5.5]{exampleneck} to our setting: 
\begin{corollary}\label{morsefunction}
Let $(\R^{n+1},g(t))_{0\leq t < T}$ be a rotationally symmetric $\epsilon$-asymptotically flat Ricci flow evolving from $g_{0}$, for some $\epsilon > 0$ and $T < \infty$. Then the number of necks is nonincreasing; in particular, all necks and bumps are nondegenerate except when they annihilate each other.
\end{corollary} 
\begin{proof}
According to Lemma \ref{sturmiantheoremneckpin}, given $T < \infty$ we get
\[
\lvert 1 - \phi_{s}^{2}\rvert (x,t) \leq \frac{C}{\phi^{\frac{\epsilon}{2}}}(x,t)
\]
\noindent for some constant $C > 0$ and for any $x\geq \rho$. Since $\phi(x,0)\rightarrow \infty$ at infinity and the curvature is bounded by $1$ outside $B(\origin,\rho)$ uniformly in $[0,T)$ we may choose $x_{0} > \rho$ large enough such that $\phi_{s}(x,t) \geq \beta > 0$ for any $x\geq x_{0}$ and for any $t\in [0,T)$ with $\beta$ as close as we ask to $1$. We can then apply the Sturmian theorem \cite[Theorem D]{nodal} to the evolution equation of $\phi_{s}$ in $B(\origin,x_{0})\times [0,T)$ and conclude that the number of zeroes is nonincreasing in time and drops whenever $\phi_{s}$ has a multiple zero. 
\end{proof}
Suppose $\phi_{0}$ has one isolated maximum at $x_{\ast}(0)$ and one isolated minimum at $y_{\ast}(0)$; we denote the radius of the bump by $\phi_{\text{max}}(0)\doteq \phi(x_{\ast}(0),0)$ and the radius of the neck by $\phi_{\text{min}}(0) \doteq \phi(y_{\ast}(0),0)$. The ratio $r \doteq \phi_{\text{max}}(0)/\phi_{\text{min}}(0)$ provides then a measure for the initial pinching of the neck-like region.  
In the following we let $x_{\ast}(t)$ and $y_{\ast}(t)$ denote the radial coordinates of the bump and of the neck along the flow respectively. 
\\By Corollary \ref{morsefunction}, $\phi(\cdot,t)$ is a Morse function except at the time where the bump and the neck annihilate each other; thus both $\phi_{\text{max}}$ and $\phi_{\text{min}}$ are smooth functions of time until either they become equal or the flow develops a singularity. 
\\ The main idea for proving the existence of local Type-I singularities consists in choosing the initial ratio between $\phi_{\text{max}}$ and $\phi_{\text{min}}$ (i.e. the initial pinching of the neck) larger than some lower bound depending on the scale invariant difference between spherical sectional curvature and radial sectional curvature so that the radius of the neck vanishes at some finite time before that of the bump does. Namely, we aim to prove the following: 
\begin{theorem}\label{asymptoticsscrittobene}
Let $(\R^{n+1},g(t))$, with $n\geq 2$, be the maximal solution to the Ricci flow evolving from an asymptotically flat rotationally symmetric metric $g_{0}$ containing a neck region $(x_{\ast}(0),y_{\ast}(0))\times S^{n}$. Assume that $\emph{Ric}_{g_{0}} > 0$ on the closed Euclidean ball $B(\origin,x_{\ast}(0))$ and that $R_{g_{0}}\geq 0$ on $\R^{n+1}$. Let $\beta$ be defined as
\[
\beta \doteq \inf_{\R^{n+1}}\phi_{0}^{2}(L_{g_{0}}-K_{g_{0}}),
\]
\noindent and let $r > 0$ satisfy
\[ 
r^{2} > \frac{n+1-2\beta}{n-1} + 1.
\]
\noindent If $\phi_{0}(x_{\ast}(0))\geq r \phi_{0}(y_{\ast}(0))$ then the following are satisfied:
\begin{itemize}
\item[(i)] The Ricci flow solution develops a local Type-I singularity at some $T < \infty$. 
\item[(ii)] If we set $\sigma\doteq S/\sqrt{T-t}$, with $S$ the distance from the neck, we can write the following cylindrical asymptotics for some uniform constants $C > c > 0$:
\[
\frac{\phi}{\sqrt{2(n-1)(T-t)}} \leq 1 + C\frac{\sigma^{2}}{\lvert \log(T-t)\rvert}
\]
\noindent for $\lvert \sigma \rvert \leq c\sqrt{\lvert\log(T-t)\rvert}$, and
\[
\frac{\phi}{\sqrt{2(n-1)(T-t)}} \leq C\frac{\lvert\sigma\rvert}{\sqrt{\lvert\log(T-t)\rvert}}\sqrt{\log\left(\frac{\lvert\sigma\rvert}{\lvert\log(T-t)\rvert} \right)}
\]
\noindent whenever $c\sqrt{\lvert\log(T-t)\rvert}\leq\lvert\sigma\rvert\leq (T-t)^{-\frac{\varepsilon}{2}}$, for $\varepsilon\in (0,1)$.
\end{itemize}
\end{theorem}
We first restate a result proved in \cite{exampleneck}; we omit the proof because it dose not require modifications.
\begin{lemma}[Angenent and Knopf \cite{exampleneck}, Lemma 5.6]\label{lemma1copied}
Let $g(t)$ be a Ricci flow defined on $[0,T)$ starting at $g_{0}$ as above. If $\phi_{ss}(x,0)\leq 0$ for $0\leq x\leq x_{\ast}(0)$ then $\phi_{ss}(x,t) \leq 0$ for all $0\leq x\leq x_{\ast}(t)$ for any $t\in [0,T)$.
\end{lemma} 
We note that given $g_{0}$ as above the condition $\phi_{ss}(x,0)\leq 0$ for $0\leq x\leq x_{\ast}(0)$ is equivalent to requiring the Ricci tensor of $g_{0}$ to be nonnegative on the Euclidean ball centred at the origin of radius $x_{\ast}(0)$.
\\ From Lemma \ref{sturmiantheoremneckpin} we deduce that we can apply Lemma \ref{crucialestimateinterior} to the Ricci flow $g(t)$ and thus obtain that the estimate \eqref{crucialestimate} holds on $\R^{n+1}\times [0,T)$. Therefore, as shown in Lemma \ref{existlimit}, we have $\lvert(\phi^{2})_{t}\rvert \leq C$ uniformly in $[0,T)$; we get that the following limit exists
\[
D \doteq \lim_{t\nearrow T}\phi(x_{\ast}(t),t) = \lim_{t\nearrow T}\phi_{\text{max}}(t).
\] 
\noindent If $D > 0$ then there is no singularity forming around the origin; that was again proved in \cite{exampleneck} using Lemma \ref{lemma1copied} and the estimate \eqref{crucialestimate}, which in particular implies that $A= \phi^{2}(L-K)$ is uniformly bounded. As above, we only adapt the statement to the current setting.
\begin{lemma}[Angenent and Knopf \cite{exampleneck}, Lemma 7.2]\label{lemma2copied}
If $D>0$ the cap $((0,x_{\ast}(t))\times S^{n})\cup\{\origin\}$ stays smooth.
\end{lemma}
 We are now ready to address the proof of (i) of Theorem \ref{asymptoticsscrittobene}.
\begin{proof}[Proof of (i) of Theorem \ref{asymptoticsscrittobene}] 
We first note that the scalar curvature is positive for any $t\in(0,T)$. From \eqref{asymptoticflat} we deduce that $A(x,t)\rightarrow 0$ at infinity for any time \cite{woolgar}; therefore $A(\cdot,0)$ has a nonpositive finite infimum $\beta \leq 0$. According to Lemma \ref{lowerboundApersists} such lower bound $\beta$ is preserved along the flow. 
\\Suppose that the neck disappears at some time $T^{\prime}\in (0,T)$; from the discussion above it follows that $\phi_{\text{max}}$ and $\phi_{\text{min}}$ are hence smooth functions in $[0,T^{\prime})$. By the implicit function theorem we find that the evolution equation of $\phi_{\text{min}}$ is given by (see also \cite[Lemma 6.1]{exampleneck}):
\begin{equation}\label{tointegrate}
\dot{\phi}_{\text{min}}(t) = \phi_{t}(y_{\ast}(t),t) = \phi_{ss}(y_{\ast}(t),t) - \frac{n-1}{\phi_{\text{min}}}.
\end{equation}
\noindent Since the scalar curvature is nonnegative we can bound the right hand side by 
\[
\dot{\phi}_{\text{min}}(t) \leq -\frac{n-1}{2\phi_{\text{min}}},
\]
\noindent which upon integration yields
\begin{equation}\label{evolutionminimum}
\phi_{\text{min}}^{2}(t)\leq \phi_{\text{min}}^{2}(0) - (n-1)t.
\end{equation}
\noindent From \eqref{evolutionminimum} we deduce that $T^{\prime} < \phi_{\text{min}}^{2}(0)/(n-1)$ otherwise the solution becomes singular before it loses its neck. We can similarly estimate the evolution equation of $\phi_{\text{max}}$ using the lower bound for $A$; we obtain
\[
\phi_{\text{max}}^{2}(t) \geq \phi_{\text{max}}^{2}(0) + 2(\beta - n)t.
\]
\noindent We finally conclude that for any $0\leq t < T^{\prime} < \phi_{\text{min}}^{2}(0)/(n-1)$ we have
\[
\phi_{\text{max}}^{2}(t) - \phi_{\text{min}}^{2}(t) \geq \phi_{\text{max}}^{2}(0) - \phi_{\text{min}}^{2}(0) -(n-2\beta +1)t \geq \phi_{\text{min}}^{2}(0)\left(r^{2} - 1 - \frac{n - 2\beta + 1}{n-1} \right).
\]
\noindent By choosing the pinching of the neck $r$ as in the statement of Theorem \ref{asymptoticsscrittobene} we obtain that the difference between $\phi_{\text{max}}(t)$ and $\phi_{\text{min}}(t)$ stays bounded away from zero until $\phi_{\text{min}}^{2}(0)/(n-1)$.
Since the radius of the bump is positive until $\phi_{\text{min}}^{2}(0)/(n-1)$, by Lemma \ref{lemma2copied} the cap around the origin stays smooth. In particular from the proof of Lemma \ref{lemma2copied} in \cite{exampleneck} it follows that there exists a radial coordinate $x_{1} > 0$ such that the curvature is uniformly bounded in the Euclidean ball $B(\origin,x_{1})$. Since by Lemma \ref{sturmiantheoremneckpin} we can apply Lemma \ref{crucialestimateinterior} and hence obtain that \eqref{crucialestimate} holds in $\R^{n+1}\times [0,T)$, we get 
\begin{equation}\label{type1bound}
\sup_{\R^{n+1}}\lvert \text{Rm}_{g(t)}\rvert_{g(t)}\leq \frac{C}{\phi_{\text{min}}^{2}},
\end{equation}
\noindent for any time sufficiently close to $\phi_{\text{min}}^{2}(0)/(n-1)$. Therefore the neck persists until some maximal time $\hat{t}\leq\phi_{\text{min}}^{2}(0)/(n-1)$ when it collapses while the radius of the bump stays positive, which implies that the flow develops a singularity before the neck may disappear. In particular, by integrating \eqref{tointegrate} using the condition $R_{g(t)}\geq 0$ and the crude estimate $\phi_{ss}(y_{\ast}(t),t)\geq 0$ we get
\[
(n-1)(T-t)\leq \phi^{2}_{\text{min}}(t)\leq 2(n-1)(T-t)
\]
\noindent which along with \eqref{type1bound} show that the singularity is Type-I.
\end{proof}
\subsection{Convergence to shrinking cylinders.}
In this subsection we adapt the analysis in \cite{IKS1} to show that the asymptotically flat Type-I flows constructed above satisfy the same cylindrical asymptotics.
\\By the Sturmian theorem we can define a time-dependent neighbourhood of the neck as follows:
\[
\Omega \doteq \left \{ \phi_{ss}\log\left(\frac{\phi}{\delta}\right) < 0 \right\}
\]
\noindent for some $\delta > 0$ sufficiently small. We need a preliminary estimate, which in the asymptotically flat case holds on the entire space-time.
\begin{lemma}
Let $(\R^{n+1},g(t))_{0\leq t < T}$ be a rotationally symmetric $\epsilon$-asymptotically flat Ricci flow with $T < \infty$. There exists $\alpha > 0$ such that 
\[
\frac{1}{\phi}(\phi_{s}^{2}-1)\leq \alpha,\,\,\,\,\,\,\,\,\,\,\phi_{ss}\phi\log(\phi) \geq -\alpha > -\infty,
\]
\noindent uniformly in $\R^{n+1}\times [0,T)$.
\end{lemma}
\begin{proof}
Define $f \doteq (\phi_{s}^{2}-1)/\phi$. From \eqref{boundaryconditions} and Lemma \ref{sturmiantheoremneckpin} we derive that if $f$ is not uniformly bounded from above in $[0,T)$ then for any large value $M$ there exists a first maximum in the space-time such that $f(x_{0},t_{0}) = M$. A simple computation gives
\[
f_{t}(x_{0},t_{0}) \leq -\frac{(\phi_{s}^{2} - 1)^{2}}{2\phi^{3}} + \frac{(\phi_{s}^{2} - 1)}{\phi^{3}}\left(\phi_{s}^{2}(-2(n-1) -1) + n-1\right). 
\]
\noindent Since $\phi_{s}^{2}(x_{0},t_{0}) > 1$ we get
\[
f_{t}(x_{0},t_{0}) < -n\frac{M}{\phi^{2}} < 0,
\]
\noindent which proves the first inequality.
\\For the second estimate we proceed similarly. The boundary conditions and the asymptotic flatness imply that $\psi \doteq \phi_{ss}\phi\log(\phi)$ is uniformly bounded from below at the origin and at spatial infinity. Let $(x_{0},t_{0})$ be the first minimum point such that $\psi(x_{0},t_{0}) = -M$ for some large $M$; by  the estimate \eqref{crucialestimate}, which holds in the asymptotically flat case by Lemma \ref{sturmiantheoremneckpin}, we can assume without loss of generality that $\phi(x_{0},t_{0}) < 1$ and hence that $\phi_{ss}(x_{0},t_{0}) > 0$. A long but straightforward computation yields
\begin{align*}
\psi_{t}(x_{0},t_{0}) &\geq \frac{1}{\phi}\left( -2\psi\phi_{ss} + \phi_{ss}\log(\phi)\phi_{s}^{2}(-4n + 8) + \phi_{ss}(-(n-1) + 4\phi_{s}^{2})\right) \\ &+ \frac{1}{\phi}\left(2\phi_{s}^{2}\frac{\phi_{ss}}{\log(\phi)} + \log(\phi)\frac{1-\phi_{s}^{2}}{\phi}\phi_{s}^{2}(-2(n-1))\right). 
\end{align*}
\noindent By using that $(\phi_{s}^{2}-1)/\phi$ is uniformly controlled from above we can bound the right hand side from below as follows:
\[
\psi_{t}(x_{0},t_{0}) \geq \frac{\phi_{ss}}{\phi}\left( 2M -n + 1 - 2\frac{\alpha}{\lvert\log(\phi)\rvert} - \frac{\alpha \lvert\log(\phi)\rvert}{\phi_{ss}}\right).
\]
\noindent Since by \eqref{crucialestimate} $\phi\phi_{ss}$ is uniformly bounded, by taking $M$ large we make $\phi$ as small as we need; therefore $\lvert\log(\phi)\rvert/\phi_{ss} = M^{-1}\phi(\log(\phi))^{2}$ is small and the right hand side is positive for $M$ large enough. 
\end{proof}
Since the neck is shrinking at a Type-I rate one can rewrite the second inequality in the previous Lemma in a more geometric way.
\begin{corollary}\label{corollarydoggo}
In the region $\Omega$ where $K \leq 0$ the following holds:
\[
(T-t)\lvert K \rvert \leq \frac{C}{\lvert \log(T-t)\rvert}.
\]
\end{corollary}
One can then rely on the argument in \cite[Lemma 16]{IKS1} which extends to our setting up to taking $f = g$ and generalizing it to any dimension $n + 1 \geq 3$. That completes the proof of (ii) of Theorem \ref{asymptoticsscrittobene}.
\subsection{Initial data leading to Type-I neckpinches.}
In this subsection we sketch how by adapting the analogous construction in \cite{exampleneck} one can find initial data satisfying the assumptions of Theorem \ref{asymptoticsscrittobene}. We explicitly consider the case $n\geq 4$; the cases $n =2,3$ only require a further smoothing step where by perturbing the metrics described in \cite{exampleneck} one can obtain initial data that are again asymptotically flat.
Given $0 < \alpha < 1$ we define
\[
f:x\mapsto \begin{cases} 
      \sin(x) & 0\leq x\leq x_{\alpha} \\
      W_{\alpha}(x) \equiv \sqrt{\alpha + (x-\frac{\pi}{2})^{2}} & x\geq x_{\alpha}
      \end{cases}
\]
\noindent where $x_{\alpha}$ is the unique intersection between $W_{\alpha}$ and $\sin$ in $(0,\pi/2)$. By the analysis in Section 8 of \cite{exampleneck} we can smooth $f$ so that we obtain a radius $\phi$ that satisfies the following properties:
\begin{itemize}
\item[(i)] $\phi(x)$ coincides with $\sin(x)$ in a radial neighbourhood of the origin so that the rotationally symmetric metric $g_{0}= (dx)^{2} + \phi^{2}(x)\hat{g}$ is a smooth metric on $\R^{n+1}$ by \eqref{boundaryconditions}.
\item[(ii)] $R_{g_{0}}\geq 0$ by \cite[Lemma 8.1]{exampleneck}.
\item[(iii)] There exists $\tilde{\alpha} > 0$ such that for any $\alpha < \tilde{\alpha}$ $g_{0}$ has a bump of radius $\phi_{\text{max}}(0)\geq \tilde{\alpha}$ and has a neck of radius $W_{\alpha}(\pi/2) = \sqrt{\alpha}.$
\item[(iv)] The scale invariant quantity $A_{g_{0}}$ satisfies $A_{g_{0}}\geq \beta$ with $\beta$ independent of $\alpha$.
\item[(v)] By direct computation one can check that $g_{0}$ satisfies \eqref{asymptoticflat} for any $\epsilon \leq 2$ and hence is asymptotically flat.
\end{itemize}
\noindent Therefore, by choosing $\alpha$ small enough we obtain a metric $g_{0}$ for which Theorem \ref{asymptoticsscrittobene} applies; equivalently, the Ricci flow evolving from $g_{0}$ develops a local Type-I singularity where the radius of the neck vanishes. 
\subsection{Immortal Ricci flows with an initial neck.} It remains to show that there exist rotationally symmetric (asymptotically flat) Ricci flows on $\R^{n+1}$ that have an initial mild neck which then disappears in finite-time. The evolution equations of the scalar curvature and the scale-invariant quantity $A$ prevent us from deriving general conditions which may describe the situation where the neck is \emph{sufficiently mild} similarly to Theorem \ref{asymptoticsscrittobene}; equivalently, it seems much harder to check whether the assumptions in Theorem \ref{asymptoticsscrittobene} are in some sense optimal. However we are still able to show that there exist examples of Ricci flows evolving an initial metric with a neck to a metric with no minimal embedded hyperspheres in finite time.
\\We recall that given two Riemannian metrics $g_{1}$, $g_{2}$ we say that $g_{1}$ is $\varepsilon$-close to $g_{2}$ if 
\[
(1+\varepsilon)^{-1}g_{2}\leq g_{1}\leq (1 + \varepsilon)g_{2}.
\]
The following argument heavily relies on \cite{schulze}.
\begin{proof}[Proof of Proposition \ref{propositionlercia}]
Let $\varepsilon = 1$ and choose $\varepsilon_{0}$ to satisfy \cite[Theorem 1.2]{schulze}. The existence of an initial neck-like region is compatible with $g_{0}$ being $\varepsilon_{0}$-close to the Euclidean metric; in particular, the difference between the radius of the bump and the radius of the neck is smaller than $2\varepsilon_{0}$. Corollary \ref{morsefunction} then implies that the critical values of the radius $\phi$ are smooth functions of time; by the implicit function theorem and the evolution equation \eqref{equationphi} we find
\[
\phi_{\text{max}}\dot{\phi}_{\text{max}}(t)\leq - (n-1).
\]
\noindent We conclude that the neck must disappear at some finite time $\hat{t} < \phi_{\text{max}}^{2}(0)/2(n-1)$, otherwise the flow would develop a finite-time singularity hence contradicting that $g(t)$ is immortal due to \cite[Theorem 1.2]{schulze}.
\end{proof}
\begin{remark}
We point out that in the examples given by Proposition \ref{propositionlercia} one can control from above the time elapsed along the flow before the neck disappears by $\phi_{\text{max}}^{2}(0)/2(n-1)$. 
\end{remark}
\bibliographystyle{abbrv}
\bibliography{references_rotational}

\end{document}